\documentclass[12pt,leqno]{article}
\usepackage{amssymb,amsmath,amscd}
\usepackage{pb-diagram}
\usepackage{graphics}
\usepackage{fancyhdr}
\usepackage[all]{xy}
\usepackage{amsthm}
\usepackage{mathtools}
\usepackage{latexsym}
\usepackage{bbm}

\newtheorem{theorem}{Theorem}[section]
\newtheorem{lemma}[theorem]{Lemma}

\newtheorem{proposition}[theorem]{Proposition}
\newtheorem{corollary}[theorem]{Corollary}
\newtheorem{remark}[theorem]{Remark}
\newtheorem{claim}[theorem]{Claim}
\newtheorem{definition}[theorem]{Definition}
\newcommand{\ff}{\textbf{f}}
\newcommand{\hh}{\textbf{h}}
\newcommand{\lc}{\textbf{l}^2}
\newcommand{\lcp}{\textbf{l}}
\renewcommand{\aa}{\textbf{a}}

\newcommand{\xx}{\textbf{x}}
\newcommand{\oo}{\textbf{0}}
\newcommand{\yy}{\textbf{y}}
\newcommand{\zz}{\textbf{z}}

\newcommand{\cc}{\textbf{c}}
\newcommand{\ra}{\rightarrow}

\newcommand{\R}{\mathbb{R}}
\newcommand{\RN}{\mathbb{R}^d}
\newcommand{\lr}{\textbf{l}^2}
\newcommand{\Z}{\mathbb{Z}}
\newcommand{\N}{\mathbb{N}}
\newcommand{\la}{\lambda}
\newcommand{\Id}{\mathrm{Id\,}}
\newcommand{\Ker}{\mathrm{Ker\, }}
\newcommand{\Coker}{\mathrm{Coker\,}}
\newcommand{\ind}{\operatorname{ind}}

\newcommand{\hof}{\hbox{\,of\, }}
\newcommand{\Ind}{\mathrm{Ind\, }}
\newcommand{\rk}{\operatorname{rk}}
\newcommand{\sk}{\medskip}
\newcommand{\hfor}{\hbox{ \,for\, }}
\newcommand{\hif}{\hbox{ \,if\, }}

\newcommand{\cB}{\mathcal{B}}
\newcommand{\cT}{\mathcal{T}}
\newcommand{\cS}{\mathcal{S}}
\newcommand{\cC}{\mathcal{C}}
\newcommand{\cO}{\mathcal{O}}

\newcommand{\hand}{\hbox{ \,and\, }}
\newcommand{\im}{\operatorname{Im}}
\newcommand{\sign}{\operatorname{sign}}
\newcommand{\ö}{^}

\title{Global bifurcation of homoclinic trajectories of discrete dynamical
systems}

\author{Jacobo Pejsachowicz
\footnote{Dipartamento di Matematica, Politecnico di Torino, Corso Duca Degli Abruzzi 24, 10129 Torino,
Italy, e-mail: jacobo.pejsachowicz@polito.it } \&
        Robert Skiba\footnote{Faculty of Mathematics and Computer Science,
        Nicolaus Copernicus University, Chopina 12/18, 87-100 Torun, Poland, e-mail: robo@mat.umk.pl}}

\begin{document}
\maketitle
\begin{abstract}
We prove the existence of an unbounded connected
branch of nontrivial homoclinic trajectories of a family of
discrete nonautonomous asymptotically hyperbolic systems
parametrized by a circle under assumptions involving the
topological properties of the asymptotic stable bundles.

\end{abstract}

\section{Introduction }

It was shown in \cite{Pej-Ski} that the different twisting of the asymptotic stable
bundles at  plus and minus infinity of a family of discrete,
nonautonomous, asymptotically hyperbolic systems parametrized by a
circle, leads to the appearance of homoclinic trajectories bifurcating from the trivial branch of
stationary solutions.
\newline\indent In \cite{Pej-Ski}, only small homoclinic trajectories close to the
stationary branch were found. Here  we will  improve our previous results for  the  problem.
 It turns out that the same topological condition supplemented
with other listed below, which ensure the properness of the
nonlinear operator naturally associated to the problem, allows us
to establish the existence of  homoclinic trajectories of
arbitrarily large  norm. What is more, using  the  global
bifurcation theory  of \cite{Pej-Rab}, we show the existence of
connected branches of nontrivial homoclinics going from  the
stationary branch to infinity.
\newline \indent  Much as in our previous paper we will translate  the appearance of
homoclinic trajectories  into a problem of  bifurcation of  zeros
of a parametrized family of $C^1$-Fredholm maps $G\colon S^1\times
X\to X$, where $X$ is a function space naturally associated to the
problem. Our approach uses  a peculiarity of the topological
degree for proper Fredholm maps of index zero. Namely, that it  is
preserved only up to sign under  homotopies of Fredholm maps
(\cite{Fi-Pej-Rab,Pej-Rab}). As a matter of fact, that the degree
can change sign
 along a homotopy will be central to our arguments. To say it shortly, we turn the lack of
homotopy invariance of the degree of Fredholm maps into  an useful
instrument  for the analysis  of bifurcation phenomena.
\newline \indent  We refer to \cite{Potz,Potz-1,Potz-2, Potz-3, Ras}
for other results on bifurcation of bounded solutions of
difference equations, however to our best knowledge both the
results and the methods of this paper are novel. The relation
between the topological properties of the asymptotic bundles and
bifurcation of homoclinics is far more subtle than the classical
spectral analysis at the potential bifurcation points. Even for
the simplest topologically nontrivial parameters space $S\ö1$ it
requires  topological instruments which at a first glance may
appear unfamiliar to many. However  we believe that  the interest
of the result and the generality of the method in proof provides
enough reasons for its introduction.
\newline \indent The paper is organized as follows. In Section $2$  we introduce the problem and the basic invariant
measuring the topological nontriviality of the asymptotic
bundles. Then we state our main theorem about the existence of a
connected branch of nontrivial homoclinic solutions which connects
the stationary branch to infinity through homoclinics of
arbitrarily large norm.
\newline\indent In Section $3$ we state and prove an index theorem for families  of linear
Fredholm operators  which will be used  in order to show that the
twisting of asymptotic bundles forces the appearance of homoclinic
trajectories (see also \cite{Ba, Sa}). The proof  is similar to
the one given in \cite{Pej-Ski} except for the fact that (in order
to ensure properness of the relevant map) we have to work in a
different function space (see also \cite{Pejs-2, Pejs-3, Pejs-4}).
Section $4$ is devoted to show that $G$ is a continuous family of
$C^1$-Fredholm maps. The continuity and smoothness of  $G$
involves only standard arguments,  many of them taken from
\cite{Potz-2, Potz-3}. The Fredholm property is derived from the
asymptotic hyperbolicity of the linearization at the stationary
solution. In order to apply the global bifurcation theory  for
$C\ö1$-Fredholm maps  the map $G$ has to be proper on closed
bounded sets. Using ideas from \cite{Se-St} we will prove
properness of $G$ in Section $5$ (see also \cite{Mo}).  In Section
$6$ we discuss the generalized homotopy property of the
topological degree constructed in \cite{Pej-Rab} and we prove our
main theorem using the computation of the index bundle from
Section $3$ (see also \cite{Bar}). Section $7$ contains a
nontrivial example illustrating our result.

\section{The main theorem}

All considered topological spaces are metric and all single-valued
maps between spaces are continuous. Given a normed space
$(\mathbb{E},\|\cdot\|)$, $\bar B(x,r),$ and $B(x,r)$, will denote
the closed and open disk centered at $x$ of radius $r$
respectively. The Euclidean norm in $\R^d$ is denoted by
$|\cdot|$; $\bar B_d(x,r)$ (resp. $B_d(x,r)$) is the closed (resp.
open) disk centered at  $x\in\R^d$; $d\geq 1$, of radius $r$.
Additionally, throughout the article, a norm of a matrix $M$ will
be denoted by $|M|$.
\newline\indent
A nonautonomous discrete dynamical system on $\R^d$ (or   is defined by
a doubly infinite sequence of  maps \begin{equation}\label{-1}\ff=\{f_n\colon \R^d\ra\R^d
\mid n\in \Z\}.\end{equation}
 A sequence $\xx=(x_n)$ which solves the equation 
\begin{equation}\label{1}
x_{n+1}=f_n(x_n).
\end{equation}
is called  trajectory of the system.  
A constant trajectory of $\ff$  is called  {\it stationary}. 
In \cite{Potz} \eqref{1} is termed as  "nonautonomous
difference equation", its  solutions  are trajectories of the
corresponding  dynamical system.\newline\indent

 In what follows we always assume that each $f_n$ is at least $C^1$ and   that $f_n(0)=0.$ In this case the sequence
$\oo =(0_n) $ is  a stationary  trajectory of $\ff$. A trajectory
$\xx=(x_n)$ of $\ff$ is called {\it homoclinic} to $\oo$, or
simply a homoclinic trajectory, if $\lim\limits_{n\ra
\pm\infty}x_n=0$. The stationary solution $\oo$ is trivially
homoclinic to itself. Here we will be interested in nontrivial
trajectories homoclinic to $\oo.$
 Every  homoclinic
trajectory of $\ff$ belongs to the space
\begin{equation*}
{\bf c}(\R^d):=\left\{\xx\colon \Z\ra \R^d \;\;\Big|\;\;
\lim_{|n|\ra \infty } x_n =0\right\}
\end{equation*}
equipped with the norm $\|\xx\|_{\infty}:=\sup\limits_{n\in
\Z}|x_n|$ (see \cite{Pej-Ski}). However in this paper we restrict
ourselves to the following space
\begin{align*}
{\bf l}^2:&=\left\{\xx\colon \Z\to \R^d\;\;\Big|\;\;
\|\xx\|:=\left(\sum_{n\in
\Z}|x_n|^2\right)^{1/2}<\infty\right\}\subset {\bf c}(\R^d).
\end{align*}

The  value $\xx(n)=x_n$  of an element $\xx \in \lc$ for $n\in \Z$
will be denoted also by $p_n(\xx)$ according to the convenience.

\sk
We will show
below that, under appropriate assumptions on the dynamical system
$\ff$, the Nemytskii (substitution) operator $F\colon \lc\ra \lc$
given by
\begin{equation}\label{2}
F(\xx)=(f_n(x_n))
\end{equation}
is a well defined  $C^1$-map verifying $F(\oo)= \oo.$ In this way
nontrivial homoclinic trajectories become the nontrivial solutions
of the equation $S\xx-F(\xx) =\oo,$ where $S\colon \lc\ra \lc$ is
the shift operator  given by
\begin{equation}\label{shift-S}
S\xx = (x_{n+1}).
\end{equation}

It should be noted that we have  considered  the space $\lc$ here
because on this space we are able to provide simple conditions
which are necessary and sufficient for $S-F$  to be
proper on closed bounded subsets of $\lc.$
\newline\indent The linearization of the system $\ff$ at the stationary solution
$\oo$ is the nonautonomous linear dynamical system $\aa\colon
\Z\times \R^d\ra\R^d$ defined by the sequence of matrices
$(a_{n})\subset \R^{d\times d},$ with $a_n = Df_n(0)$. The
corresponding linear difference equation is defined by
\begin{equation}\label{1'}x_{n+1} = a_{n}x_n.
\end{equation}

We will deal only with discrete nonautonomous dynamical systems
whose linearization at $\oo$ is asymptotic  for $n \ra \pm \infty$
to an autonomous linear hyperbolic dynamical system $\aa$ (i.e.,
verifying $a_n =a$ for all $n\in \Z$, where $a$ is a hyperbolic matrix).
We will call systems with the above property {\it asymptotically
hyperbolic}.

 An
invertible matrix  $a$ is called {\it hyperbolic} if $a$ has no
eigenvalues of norm one. The spectrum $\sigma(a)$ of a hyperbolic
matrix $a$ consists of  two disjoint closed subsets
$ \sigma^s =\sigma(a)\cap\{|z| < 1\}$ and $\sigma^u=\sigma(a)\cap\{|z|> 1\}$, which implies that 
$\R^d$ has an $a$-invariant spectral decomposition $\R^d=E^s(a)
\oplus E^u(a)$, where $E^s(a) $ (respectively $E^u(a)$)  is the
direct sum real parts of  the generalized eigenspaces
corresponding to eigenvalues of $a$ inside of the unit disk
(respectively outside of the unit disk).

  The spaces  $E^{s/u}(a)$ are the stable and unstable subspace of the autonomous dynamical system associated to $a$ respectively,  since$ \zeta \in E^s(a)$ if and only if
$\lim\limits_{n\ra\infty}a^n\zeta=0,$ while $\zeta\in E^u(a)$
if and only if $\lim\limits_{n\ra\infty}a^{-n}\zeta=0.$\\\\
A {\it  $C^1$-family of dynamical systems parametrized by
the unit circle $S^1$} is defined by a sequence of maps
\begin{equation}\label{cfd}
\ff=\left\{f_n\colon S^1\times \R^d \ra \R^d \mid n\in\Z\right\}
\end{equation}
such that $f_n$ is $C^1,$ for all $n\in \Z.$

In what follows we will also assume everywhere that
$f_{n}(\la,0)=0$, for all $\lambda\in S^1$ and $n\in \mathbb{Z}$.
We will use $\ff_\la$ to denote the dynamical system corresponding
to the parameter value $\la.$ We will use $\ff_\la$ to denote the
dynamical system corresponding to the parameter value $\la.$
Alternatively one can think of $\ff$ as a double infinite sequence
of  $C^1$-maps $f_n \colon [a,b]\times \R^d \ra \R^d$ such that
$f_n(a,-)$ coincides with $f_n(b,-)$ up to the first order.

We will say that $(\la, \xx)$ is a {\it homoclinic solution} for
the family $\ff$ if $(\la,\xx)$ solves the parameter-dependent
difference equation:
\begin{equation}\label{main-system}
x_{n+1}=f_n(\la,x_n),\;\text{ for all }n\in\Z,
\end{equation} or equivalently,  if  $\xx=(x_n)$ is a homoclinic trajectory of the dynamical system $\ff_\la.$
Homoclinic solutions of (\ref{main-system}) of the form $(\la,{\bf
0})$ are called trivial and the set $S^1\times\{\bf 0\}$ is called
the {\it trivial or stationary branch.}

Our aim is to  show how the topology of  the parameter space
$S^1$, on which  our dynamical system depends,  forces  the
appearance of branches of homoclinic  trajectories  connecting
small  homoclinic trajectories to the arbitrarily large ones. For
this  we will apply the global bifurcation theory  for families of
$C^1$-Fredholm maps established in \cite{Pej-Rab} to the family
defined by
\begin{equation}\label{gem}
G(\la,\xx) = S \xx - F(\la,\xx),
\end{equation}
where $ F\colon S^1\times \text{$\lc$}\ra \lc$ is  the
parametrized substitution (Nemytskii)  operator
$F(\la,\xx):=(f_n(\la,x_n)).$

We will assume that the family of discrete dynamical systems
$\ff\colon  \Z\times  S^1\times \R^d\ra \R^d$ satisfies the
following conditions:
\begin{enumerate}
\item[$(A1)$] For any $M>0$ and $\varepsilon>0$ there exists $0<\delta<M$ such that for  all
$(\la_1,x_1), (\la_2,x_2) \in S^1\times \bar B_d(0,M)$  with $
d\big((\la_1,x_1),(\la_2,x_2)\big)<\delta$ one has
\begin{equation*}
\sup_{n\in\Z}\left|\frac{
\partial f_n(\la_1 ,x_1)}{\partial x}-
\frac{\partial f_n(\la_2 ,x_2)}{\partial x}\right|<\varepsilon
\text{ and } \sup_{n\in\Z}\left|\frac{
\partial f_n(\la_1,x_1)}{\partial \la}-
\frac{\partial f_n(\la_2,x_2)}{\partial \la}\right|<\varepsilon,
\end{equation*}
where $d$ is the product distance in the metric space $S^1\times
\R^d\subset \R^2\times \R^d.$
\item[$(A2)$] The family of matrices
\begin{equation*}
a_n(\la):=\displaystyle\frac{\partial f_n(\la,0)}{\partial x}
\xrightarrow[n\rightarrow\pm\infty]{} a(\lambda,\pm\infty)
\end{equation*}
(uniformly with respect to $\la\in S^1$), where
$a(\lambda,\pm\infty)$ is a hyperbolic matrix. Moreover, we assume
that for some (and hence for all) $\la\in S^1$, the limits
$a(\la,+\infty)$ and $a(\la,-\infty)$ have the same number of
eigenvalues (counting algebraic multiplicities) inside of the unit
disk.
\item[$(A3)$] There exists $\la_0\in S^1$ (say $\la_0 =1$) such  that
the following two difference equations
\begin{equation*}
x_{n+1}=f_{n}(1,x_n)\quad \text{and}\quad x_{n+1}=a_n(1)x_n
\end{equation*}
admit only the trivial solution $(x_n=0)_{n\in\Z}$. Equivalently,
$\ff(1)$ and $\aa(1)$ have no nontrivial homoclinic trajectories.
\item[$(A4)$] For any $x\in \R^d$ and $\la\in S^1$,
\begin{equation*}
f_n(\la,x)\xrightarrow[n\rightarrow\pm\infty]{}
f^{\infty}_{\pm}(\la,x)
\end{equation*}
(uniformly with respect to any bounded set $B\subset \R^d$) and
the following two difference equations
\begin{equation*}
x_{n+1}=f^{\infty}_{+}(\la,x_n)\quad \text{and}\quad
x_{n+1}=f^{\infty}_{-}(\la,x_n)
\end{equation*}
admit, for any $\la\in S^1$, only the trivial solution
$(x_n=0)_{n\in\Z}$.
\end{enumerate}

It follows easily from  $(A2)$ the map $\la \ra a(\la,\pm\infty)$ is a continuous
family of hyperbolic matrices. Moreover,  the  vector spaces
$E^s(\la,\pm\infty)$ and $E^u(\la,\pm\infty)$ whose elements are
the real parts of the generalized eigenvectors of
$a(\la,\pm\infty)$ corresponding to the eigenvalues with absolute
value smaller (respectively greater) than $1$ are fibers of a pair
of vector bundles $E^s(\pm\infty)$ and $E^u(\pm\infty)$ over $S^1$
which decompose the trivial bundle $\Theta(\R^d)$ with fiber
$\R^d$ into a direct sum:
\begin{equation}\label{dirsum}
E^s (\pm\infty)\oplus  E^u(\pm\infty) = \Theta(\R^d).
\end{equation}
The vector bundles   $E^s(\pm \infty)$  and $E^u(\pm \infty)$ will be
called  \textit{stable} and \textit{unstable}  asymptotic bundles
at $\pm\infty.$  Our main theorem relates the appearance of
homoclinic solutions to the topology of the asymptotic stable
bundles $E^s(\pm \infty)$. 
In what follows it will be convenient
to work  with the multiplicative group $\Z_2=\{1,-1\}$
instead of the standard additive $\Z_2=\{0,1\}.$ 
A vector bundle
over $S^1$ is orientable if and only if it is trivial, i.e.,
isomorphic to a product $S^1\times\R^k.$ Moreover, whether a given
vector bundle $E$ over $S^1$ is orientable  or is not is determined by
a   topological invariant $w_1(E)\in\Z_2.$

In order to define $w_1(E)$  let us identify $S^1$ with the
quotient of an interval $ I= [0,1]$ by its boundary $\partial I =
\{0,1\}.$ If $p\colon [0,1]\ra S^1= I/\partial I$  is the
projection,  the pullback bundle $p^*E =E'$ is the vector  bundle
over $I$ with fibers $ E'_t =E_{p(t)}.$  Since $I$ is contractible
to a point, $E'$ is trivial and the choice of an isomorphism
between $E'$ and the product bundle provides $E'$ with a frame,
i.e., a basis $\{e_1(t),...,e_k(t)\}$ of $E'_t $ continuously
depending on $t$. Since  $E'_0 =E_{p(0)}=E_{p(1)}=E'_1,$
$\{e_i(0)\mid 1\leq i\leq k\}$ and $\{e_i(0)\mid 1\leq i\leq k\}$
are two bases  of the same vector space. We define $w_1(E)\in\Z_2$
by
\begin{equation} \label{whitney}
w_1(E):=\sign \det C,\end{equation} where $C$ is the matrix
expressing  the basis $\{e_i(1)\mid 1\leq i\leq k\}$ in terms of
the basis $\{e_i(0)\mid 1\leq i\leq k\}$.

Clearly, 
$w_1(E)$ is independent from the choice of the frame. Moreover,
$w_1(E)=1$ if and only if $E$ is trivial. Indeed, if $E$ is a
trivial bundle, then by definition, $w_1(E)=1.$  On the other
hand, if $w_1(E)=1,$ $\det C>0$,  and there exists a path $C(t)$
with $C(0)=C$ and $C(1)= \Id.$  Then $ f_i(t) = C(t) e_i(t)$ is a
frame such that $f_i(0)= f_i(1)$ and hence $\Phi(t,
x_1,\ldots,x_k) = \left(t,\sum_{i=1}^k x_i f_i(t)\right)$ is an
isomorphism between $S^1\times \R^k$ and $E.$

Under  the isomorphism $H^1(S^1;\Z_2)\cong\Z_2$, the invariant  $w_1(E)$
can be identified with the first Stiefel-Whitney class of $E$.

A point $\la_*\in S^1$ is
a {\it bifurcation point} for homoclinic solutions of
\eqref{main-system} from the trivial  branch of stationary
solutions $ \cT_0=\{(\la,\oo)\mid \la\in S^1\}$  if in every
neighborhood of $(\la_*,\oo)$ there is a point $(\la,\xx)$ such
that $\xx$ is a nontrivial homoclinic solution of
$x_{n+1}=f_n(\la,x_n).$

Bifurcation points from infinity are defined in a similar way.
Namely, $\la_*\in S^1$ is a {\it bifurcation point  from infinity}
for homoclinic solutions of \eqref{main-system} if there is a
sequence $(\la_n,\xx_n)$ of homoclinic solutions of
\eqref{main-system} with
$\la_n\xrightarrow[n\rightarrow\infty]{}\la_*$ and
$\|\xx_n\|\xrightarrow[n\rightarrow\infty]{} \infty$. Due to the
compactness of $S^1,$ any unbounded sequence of solutions contains
a subsequence  $(\la_n,\xx_n)$ such that $\la_n$ converges to a
bifurcation point from infinity. By $B_0$ (resp. $B_{\infty}$) we
will denote the set of all bifurcation points of
\eqref{main-system} from the  trivial branch of stationary
solutions (resp. the set of all bifurcation points of
\eqref{main-system} from infinity).

In order to state our result in a more symmetric form we will
introduce the trivial branch at infinity. The one point
boundification   of a normed space $E$ is the  topological space
$E^+:=E\cup\{\infty\}$  with a base of neighborhoods of
$\{\infty\}$ given by $D\cup \{\infty\},$ where D is a complement
of a closed bounded subset of $E.$   Notice that if  $A\subset E$
is a closed and locally compact subset of $E$, then its closure
$\bar A =A\cup\{\infty\}$ in $E^+$ is the one-point
compactification $A^+ $ of $A.$ A sequence in  $ x_n\in E$ such
that $\|x_n\|\xrightarrow[n\rightarrow\infty]{}\infty$  converges
to the point $\{\infty\} $ in $E^+.$ By definition the subset
$\cT_\infty =\{ (\la,\infty) \mid \la \in S^1\}$ of $S^1\times
{\lr}^+ $  is the {\it trivial branch at $\{\infty\}$}.

The main result of this paper reads as follows:

\begin{theorem}\label{Theorem-A}
If the system \eqref{cfd} verifies $(A1)$--$(A4)$ and if
\begin{equation}\label{w1}
w_1(E^s(+\infty))\neq w_1(E^s(-\infty)),
\end{equation}
then:
\begin{itemize}
\item[\emph{[i]}] The connected  component $\cC_0$ of  $\cT_0$ in the set $\cS\subset S^1\times \emph{$\lr$} $ of
all homoclinic solutions of \eqref{main-system} is unbounded. In
particular, both $B_0$ and $B_\infty$ are nonempty.
\item[\emph{[ii]}] The set $\cS_0= \cS - \cT_0$ of all nontrivial homoclinic solutions of \eqref{main-system}
contains  a continuum $($i.e., closed connected subset $\cC)$
whose closure $\bar \cC$ in $ S^1\times \emph{$\lr$}^+ $
intersects both $\cT_0$ and $\cT_\infty.$
\end{itemize}
\end{theorem}

Therefore,  not only $B_0$ and $B_\infty$ are not empty but there
is a connected branch of nontrivial homoclinic solutions of
\eqref{main-system} connecting $\cB_0=B_0\times \{\oo\} $ to
$\cB_\infty = B_\infty \times \{\infty\}.$ The theorem will be
proved in Section 6.  The main ingredients of the proof are the
computation of the index bundle of the family of linearized
equations at the  trivial branch in terms of the asymptotic stable
bundles at $\pm \infty$ and the generalized homotopy property of
the base point degree of the family of induced Fredholm maps. The
next section  is entirely devoted to the first of the above
mentioned tools.

\section{The index bundle}
Our goal here  is to establish  the Fredholm property  of
operators induced on functional spaces by a linear asymptotically
hyperbolic systems and to compute the index bundle of  a
parametrized  family  of  such operators.\\\\
Let us shortly  recall the concept of  the  index bundle
of a family of Fredholm operators.  For a comprehensive 
presentation see \cite{Pejs-1} or \cite{At}.  By  $\mathcal{L}(X,Y)$ (resp.
$\mathcal{L}(X)$) we will denote the space of bounded linear
operators between two Banach spaces $X$ and $Y$ (resp. from $X$ into
itself).  A bounded operator
$T\in \mathcal{L}(X,Y)$ is Fredholm if it has finite dimensional kernel and
cokernel. The index of a Fredholm operator  is by definition $\ind
T:=\dim \Ker T - \dim \Coker T.$ The space of all Fredholm
operators will be denoted by $\Phi(X,Y)$ and those of index $k$ by
$\Phi_k(X,Y).$ For each $k,$ $\Phi_k(X,Y)$ is an open subset of
$\mathcal{L}(X,Y).$\\\\
The index bundle extends to families of Fredholm operators the notion  of index of a single Fredholm operator.

 Let  $\{L_\la\colon X \ra Y ; \la \in \Lambda\}$ be a family  of Fredholm operators depends continuously on a parameter $\la$ belonging to some compact  topological space $\Lambda.$  If the kernels $\Ker L_\la$ and cokernels $\Coker L_\la$  form two vector bundles $\Ker L$ and $\Coker L$ over
$\Lambda,$ then, roughly speaking, the index bundle is  the
difference $\Ker L-\Coker L,$  where one gives  a meaning to the
difference of two vector bundles  by working in  the Grothendieck group $KO(\Lambda ),$
which by definition  is the group completion of the abelian
semigroup $\text{Vect} (\Lambda )$ of all isomorphism classes of
real vector bundles over  $\Lambda $ (generalizing the fact that $\Z$ is the group completion of the semigroup $\N$).  The elements of
$KO(\Lambda)$  are called  virtual bundles.  Each virtual bundle
is  a difference $[E] - [F],$ where $E, F$ are vector bundles over
$\Lambda $ and $[E]$ denotes the corresponding element of
$KO(\Lambda ).$ One can show that $ [E] - [F]= 0$ in $KO(\Lambda
)$ if and only if the two vector bundles become isomorphic after
the addition of a  trivial vector bundle to both sides. Taking
complex vector bundles instead of the real ones leads to  the
complex Grothendieck  group denoted by $K(\Lambda).$ In what
follows the trivial bundle  with fiber $\Lambda\times V$ will be
denoted by  $\Theta(V)$ and $\Theta(\R^d)$ will be simplified to
$\Theta^d.$

For a  general family  $L \colon \Lambda \ra  \Phi(X,Y)$  neither the kernels
nor the cokernels of $L_\la$ will form a vector bundle. However,
since $\Coker L_\la$ is finite dimensional, using  compactness of
$ \Lambda,$ one can find a finite dimensional subspace $V \hof\,
Y$ such  that
\begin{equation} \label{1.1}
\hbox{\rm Im}\,L_\la+ V=Y  \ \hbox{\rm for all }\  \la \in
\Lambda.
\end{equation}

Because of the condition \eqref{1.1}  the family of
finite dimensional subspaces $E_{\la}=L_{\la}^{-1}(V)$ defines a
vector bundle over $ \Lambda$ (see \cite {Pejs-1}) with  total
space
\[E= \bigcup_{\la \in \Lambda}\, \{\la\} \times E_\la.\]
By definition, the {\it index bundle}  is the virtual bundle:
\begin{equation} \label{defind}
\Ind L= [E]-[\Theta (V)] \in KO(\Lambda).
\end{equation}

The index bundle have the same  properties as the ordinary
index. Namely, homotopy invariance, additivity with respect to
directs sums, logarithmic property under composition of operators.
Clearly it vanishes if $L$ is a family of isomorphisms. We will
mainly use in the  sequel:
\begin{itemize}
\item[$(i)$] \emph{Homotopy invariance:}  Let $H\colon[0,1]\times   \Lambda\rightarrow  \Phi (X,Y)$ be a homotopy,   then $\Ind  H_0 =\Ind  H_1.$ In particular,
$\Ind(L+K)=\Ind L,$ if $K$ is a family of compact operators.
\item[$(ii)$] \emph{Logarithmic property:}  $\Ind\bigl( LM \bigr) = \Ind L + \Ind M.$
\end{itemize}

We will mostly  work with families of Fredholm operators of index
$0.$ The index bundle of a family of Fredholm operators of index
$0$ belongs to the reduced Grothendieck group $
\widetilde{KO}(\Lambda ),$ i.e.,  the  subgroup generated by
elements $[E] -[F]$  such that $\dim E_\la =\dim F_\la.$ It can be
shown  that  any element $\eta \in \widetilde{KO}(\Lambda)$ can be
written as $[E] -[\Theta^N ].$ Moreover, $[E] -[\Theta^N ] = [E']
-[\Theta^M] $ in $\widetilde{KO}(\Lambda )$ if and only if there
exist two trivial bundles $\Theta $ and $\Theta'$ such that
$E\oplus \Theta$ is isomorphic to  $E'\oplus \Theta',$ (see
\cite[Theorem 3.8]{Hus}).
\\\newline\indent The rest of this section is devoted to the computation  of the index bundle  of the family of operators associated to a family of  linear asymptotically hyperbolic systems.

 Let  $GL(d)$be the set of all invertible matrices
in $\R^{d\times d},$ a map $\aa \colon  \Z \times S^1\ra GL(d)$  is 
a family of  linear asymptotically hyperbolic systems if 
\begin{itemize}
\item[(a)] As $n \ra\pm \infty$ the sequence  $\aa(\la)=(a_n(\la))$ converges
uniformly with respect to $\la\in S^1$ to a family of matrices
$a(\la,\pm \infty)$.
\item [(b)]  $a(\la,\pm \infty)\in GL(d)$ is hyperbolic for all $\la \in S^1.$
\end{itemize}

It is easy to see that (a) implies that $\aa_{\pm}\colon S^1\to
GL(d)$ given by $\aa_{\pm}(\la):=a(\la,\pm \infty)$ are continuous
functions of $\la.$ 

We associate to the family $\aa \colon  \Z
\times S^1\ra GL(d) $ the family of linear operators
\begin{equation*}
L = \left\{L_\la \colon \lc\ra \lc\mid \la \in S^1 \right\}
\end{equation*}
defined by $ L_\la =S-A_\la,$ where $S$ is the shift operator and
$ A_\la \colon \lc\ra \lc $  is the substitution operator $A_\la
\xx\!:=(a_{n}(\la)x_n).$  Since the  sequence $(a_n(\la))$
converges, it is bounded, from which follows immediately that
$A_\la$ and $L_\la $ are well defined bounded operators.
\begin{lemma}\label{lemma-continuity-L}
The map $A\colon S^1\ra \mathcal{L}\!\left(\emph{$\lc$}\right)$
defined by $A(\la):=A_{\la}$ is continuous with respect to the
norm topology of $\mathcal{L}\!\left(\emph{$\lc$}\right)$. Hence
the same holds for the family $L.$
\end{lemma}
\begin{proof}
Fix $\xx=(x_n)\in\lc$. Then
\begin{align*}
\|(A(\la)-A(\mu))\xx\|^2=
\sum_{n\in\Z}\left|(a_n(\la)-a_n(\mu))x_n\right|^2\leq
\sum_{n\in\Z}\left|a_n(\la)-a_n(\mu)\right|^2|x_n|^2.
\end{align*}
Furthermore, $|a_n(\la)-a_n(\mu)|\leq
|a_n(\la)-\aa_{\pm}(\la)|+|a_n(\mu)-\aa_{\pm}(\mu)|+|\aa_{\pm}(\la)-\aa_{\pm}(\mu)|.$
Fix $\varepsilon>0$. Then Assumption $(a)$ implies that there
exists $n_0>0$ such that
\begin{align*}
&|a_n(\la)-\aa_+(\la)|<\varepsilon/3 \quad \text{for all $n\geq
n_0$ and for all $\la\in S^1$},\\
&|a_n(\la)-\aa_-(\la)|<\varepsilon/3 \quad \text{for all $n\leq
-n_0$ and for all $\la\in S^1$}.
\end{align*}
Moreover, there exists $\delta>0$ such that if $d((\la,0),(\mu,0))<\delta$,
$\la,\mu\in S^1$, then
\begin{align*}
|\aa_{\pm}(\la)-\aa_{\pm}(\mu)|\leq \varepsilon/3 \text{ and }
|a_k(\la)-a_k(\mu)|\leq \varepsilon/3 \text{ for all $-n_0<k<
n_0$}.
\end{align*}
Finally, taking into account the above considerations, one obtains
that
\begin{align*}
&\|(A(\la)-A(\mu))\xx\|^2\leq
\sum_{n\in\Z}\left|a_n(\la)-a_n(\mu)\right|^2|x_n|^2\leq
\sum_{n\in\Z}\varepsilon^2|x_n|^2=\varepsilon^2\|\xx\|^2
\end{align*}
provided $d((\la,0),(\mu,0))<\delta$, which implies that $A$ is continuous
with respect to the norm topology of $\mathcal{L}(\lc)$.
\end{proof}

Clearly, $\xx=(x_n)\in \lc$ verifies a linear difference equation
$x_{n+1}=a_n(\la)x_n$ if and only if $L_{\la}\xx=0$. By the
discussion in the previous section  the families $a(\la,\pm
\infty)\in GL(d)$ define two vector bundles $E^s (\pm \infty)$
over $S^1.$ The next theorem  relates the index bundle of the
family $L$ to $E^s (\pm \infty).$

\begin{theorem}\label{prop:ind} Let
${\bf a}\colon \Z\times S^1 \ra GL(d)$ be a continuous  map
verifying $(a)$ and $(b).$ Then the family $L\colon S^1\ra
\mathcal{L}\left(\emph{$\lc$}\right)$ verifies:
\begin{itemize}
\item [$(i)$] $L_\la$ is a Fredholm operator   for
all $\lambda\in S^1$.
\item [$(ii)$] $\Ind L= [E^s (+ \infty)]-[E^s(-\infty)] \in
{KO}(S^1)$.
\end{itemize}
\end{theorem}

In particular, applying to  $\Ind L$ the rank homomorphism,   $\rk([E] -[F]) = \dim E_\la -\dim
F_\la$, we obtain
\begin{equation*}
\ind L_\la =\dim E^s (+ \infty) - \dim E^s(-\infty).
\end{equation*}
\begin{proof}
The proof  is similar to the proof of Theorem $4.1$ in
\cite{Pej-Ski}. However, since here we are working on a proper
subspace $\lr$  of ${\bf c}(\R^d)$ we have to check that our all constructions   can be done  in this subspace. Hence for the convenience of the reader we will recall the main steps of the proof pointing out the differences with \cite{Pej-Ski}. 

First of all we will show that the calculation   of the index bundle  of $L$ can be reduced to the case of the operator associated to a family of dynamical systems of a special form $\bar{\aa}$  given by
\begin{equation}\label{const}
\bar{\aa}(n,\la)=(\bar{a}_n(\la))=
\begin{cases} a(\la, +\infty) & \hif\; n\geq 0,
\\ a(\la,- \infty)
&\hif\; n<0.
\end{cases}
\end{equation}

Denoting with  $\bar A_\la \in \mathcal{L}(\lc)$
the  operator associated to $\bar{\aa}_\la,$ we
will show  that the operator  $B_\la=A_\la -\bar A_\la$ is a compact
operator for any $\la.$ 

 Let $b_n(\la)=a_n(\la)-\bar a_n(\la),$ then $B_{\la}$ is given by 
 $B_{\la}\xx=(b_n(\la) x_n).$
 Define
\begin{equation}
\label{finite} \tilde B^m_\la\xx=
\begin{cases} b_n(\la) x_n & \hif \;|n|\leq m,  \\ 0 &\hif \;
|n|>m.
\end{cases}
\end{equation}
Clearly $\im \tilde B^m_\la$ is  finite dimensional. Let us show that \begin{equation}\label{wz} \sup_{\|\xx\|=1}
\|(B_{\la}-\tilde
B^m_{\la})\xx\|\xrightarrow[m\rightarrow\infty]{}0,
\end{equation}
for $\xx\in X$. Observe that
\begin{align}\label{wz1}
\begin{split}
\|(B_{\la}-\tilde B^m_{\la})\xx\|=\sum_{|n|>m}|b_n(\la)x_n|^2 \geq
\sum_{|n|>m+1}|b_n(\la)x_n|^2=\|(B_{\la}-\tilde
B^{m+1}_{\la})\xx\|,
\end{split}
\end{align}
for all $m\in \N$. Since $\lim\limits_{|n|\ra \infty} b_n(\la)=0,$
we infer that for all $ \varepsilon>0$ there exists $n_0>0$ such
that for all $|n|> n_0$ and $\|\xx\|=1$ one has
$|k_n(\la)x_n|<\varepsilon.$ Consequently, for all $\varepsilon>0$
there exists $n_0>0$ such that
\begin{equation}\label{wz2}
\sup_{ \|\xx\|=1}\|(B_{\la}-\tilde
B^{n_0}_{\la})\xx\|\leq\varepsilon.
\end{equation}
Now taking into account \eqref{wz1} and \eqref{wz2}, we deduce
that for all $\varepsilon>0$ there exists $n_0>0$ such that for
all $m\geq n_0$ one has
\begin{equation*}
\sup_{ \|\xx\|=1} \|(B_{\la}-\tilde
B^{m}_{\la})\xx\|\leq\varepsilon,
\end{equation*}
which proves \eqref{wz}.   
Being limit of operators of finite rank, each  operator
$B_\la$  is compact.

Let $\bar L_\la =S-\bar{A}_\la.$  Then $L_\la -\bar
L_\la= B_\la$  and hence  the family $L$ differs from the family
$\bar L$ by a family of compact operators. Therefore $L_\la $ is
Fredholm if and only if  $\bar L_{\la}$ is Fredholm. Moreover being  the index bundle  invariant under compact perturbation  we have that  $\Ind \bar
L = \Ind L.$  Therefore,  in order to prove the theorem  we can assume
without loss of generality that $\aa$ has already the special form
of \eqref{const}, which we will do from now on. 

Let us introduce  two scales  $\lcp^\pm_k$ of closed subspaces of $\lc$ defined respectively  by:
\begin{align*}
\lcp^+_k:=\{\xx \in \lc \mid x_n=0 \hfor n<k\},\;\;
\lcp^-_k:=\{\xx \in \lc \mid x_n =0 \hfor n>k\}.
\end{align*}
We put   $ X^+ =Y^+ =
\lcp_0^+$ and $X^- = \lcp^-_0, \ Y^- = \lcp^-_{-1}$ and consider  
the operator   $I\colon Y^- \oplus Y^+ \ra X$ defined by $I(\xx,\yy)=\xx+\yy,$ the operator $J\colon X \ra X^- \oplus X^+$ defined  by 
\begin{equation*}
J(\xx)(n)=\begin{cases} (x_0,x_0)& \hif\; n=0,\\(x_n,0) & \hif
\;n<0,\end{cases}
\end{equation*}
and two  operators  $L^{+/-} _\lambda \colon X^{+/-}\ra Y^{+/-}$
defined respectively by

\begin{align*}
(L_\la^+\xx)(n)&=
\begin{cases} x_{n+1} -a(\la,+ \infty)x_n
& \hif n\geq 0,
\\ 0
&\hif n<0,
\end{cases}
\hand
(L^-_{\la}\xx)(n)&=
\begin{cases}
0 & \hif n>-1,
\\  x_{n+1} - a(\la,- \infty)x_n
& \hif n\leq -1.
\end{cases}
\end{align*}
 
 With the above definitions we factorize $L_\la$ through  the following
commutative  diagram:
\begin{equation}\label{eq:homdiag}
\begin{split}
\xymatrix@1{ X^{-}\oplus X^{+}\;\ar[r]^-{L^-_\lambda \oplus
L^+_\lambda} & \;Y^{-} \oplus Y^{+}  \ar[d]^{I}
\\
X \ar[u]^{J} \ar[r]^-{L_\lambda} & X.}
\end{split}
\end{equation}

Our next step is to show that $L_\la^\pm \colon X^\pm \ra Y^\pm$ are
Fredholm and  compute their index bundles.

\begin{lemma}\label{ab-lemma}  Let  $a\in GL(d)$ be a hyperbolic matrix.  Then
the  operator $S-A\colon \emph{\lcp}_0^+\ra \emph{\lcp}_0^+$
defined by
\begin{equation*}
((S-A)\emph{\xx})(n)=
\begin{cases}
x_{n+1}-ax_{n} & \emph{\hif}\; n\geq 0, \\
0 &\emph{\hif}\; n<0,
\end{cases}
\end{equation*}
is surjective with $\Ker(S-A)=\{\emph{\xx}\in \emph{\lcp}_0^+\mid
x_{n+1}=a^nx_0 \text{ for all $n\geq 0$ and } x_0\in E^s(a)\}.$
\end{lemma}

This lemma was proved in \cite[Lemma 2.1]{Ab-Ma2} for the operator
induced   in $\cc^+_0:=\{\xx \in \cc(\R^d) \mid x_i=0 \hfor i<0\}$
by constructing an explicit right inverse $R$ to the operator
$S-A\colon \cc_0^+\ra \cc_0^+$, via the convolution  with the
matrix  function $g(n)=  a^{n-1} (\mathbbm{1}_{Z^+}(n)\Id_{\R^d} -
P^u),$ where $P^u$ is the projector on the unstable subspace and
$Z^+=\{1,2,...\}$. They prove that function $g$ belongs to
$l^1(\Z,\R^{d\times d})$. But the convolution with a matrix
function in  $l^1(\Z,\R^{d\times d})$ sends $\lr$ into itself.
Hence the assertion of this lemma is also true for $S-A\colon
\lcp_0^+\ra \lcp_0^+$. For the assertion regarding the kernel is
enough to observe  that if $\xx=(x_n)\in \cc_0^+$ and
$x_{n+1}=a^nx_0$ for all $n\geq 0$ and $x_0\in E^s(a)$, then the
spectral radius theorem guarantees that $\xx\in \lcp_0^+$.\qed

By Lemma \ref{ab-lemma}
\begin{equation}\label{ker+}
\Ker L^+_{\la}=\{\xx\in X^+ \mid x_n = a(\la,+\infty)^n x_0
 \hand x_0\in E^s(\la,+\infty)\}.
\end{equation}
Hence the transformation  $\xx \mapsto x_0 $ defines an
isomorphism between $\Ker L^+ $ and $E^s(\la,+\infty),$ which is
finite dimensional. Being $\Coker L_\la =\{0\},$ $L^+_{\la}$ is
Fredholm with $\ind L^+_{\la}=\dim E^s(\la,+\infty).$ Clearly the
index bundle  $\Ind L^+ = [E^s(+\infty)].$ 

The remaining part of the proof is identical to that of Theorem $4.1$ in
\cite{Pej-Ski}.   Namely, we reduce the calculation of $\Ind L^-$ to Lemma \ref{ab-lemma}  by showing that $L^-$ is conjugated to a surjective  operator of the same form as $L^+ $  whose kernel bundle is isomorphic to 
$ E^u(\la,-\infty).$ 
Thus we  obtain that $\Ind\,L^+=
[E^s( +\infty)]$ and  $\Ind L^-= [ E^u( -\infty)].$  
Observing  that 
$I$ is an isomorphism and  and $J$ is a  monomorphic Fredholm operator of index $-d,$ from the commutativity of the diagram
 it follows that, for each $\la$, 
$L_{\la}=I(L^-_{\la}\oplus L^+_{\la})J$ is Fredholm of and
\begin{equation}\label{dim}
\begin{array}{l}
\ind(L_{\la})=
\dim  E^s(\la,+\infty)+\dim  E^u(\la,-\infty)-d=\dim
E^s(\la,+\infty)-\dim  E^s(\la,-\infty).
\end{array}
\end{equation}
Noe $(ii)$ follows from the logarithmic and  direct sum properties of the index bundle,  by  considering $I$ and $J$ as constant families 
with $\Ind\,I=0,\, \Ind\,J=-[\Theta(\R^d)].$ 

We obtain
\begin{equation*}
\Ind L = [E^u( -\infty)] +[E^s( +\infty)]-[\Theta(\R^d)]=[E^s(
+\infty)]-[E^s( -\infty)],
\end{equation*}
which proves $(ii).$
\end{proof}

\begin{remark} \label{ker} In the proof we have shown  that in the case  of systems of the special  form \eqref{const}
elements  of  $\Ker L_{\la}$ are sequences $(x_n)\in X$ such that
$x_0\in E^s(\la,+\infty) \cap E^u(\la,-\infty)$ and
$x_n=a(\la,+\infty)^{n}x_0,  \hfor  n \geq 0 \hand
x_n=a(\la,-\infty)^{n}x_0,  \hfor n \leq 0.$
\end{remark}

The obstruction  $w_1(E)$  to the triviality of a vector bundle
$E$ over $S^1$ defined in Section $2$  induces an 
isomorphism $w_1\colon \widetilde{KO}(S^1)
\ra \Z_2$ by putting
\begin{equation}\label{morewhit}
w_1([E] -[F]) = w_1(E) w_1(F).
\end{equation}  

From this and Theorem \ref{prop:ind} we obtain:
\begin{corollary}
\begin{equation}\label{w}
w_1(\Ind L)=w_1(E^s(+\infty) w_1(E^s(-\infty)).
\end{equation}\end {corollary}

\section{The continuity smoothness and the Fredholm property of the family  $G(\la,\xx)  = S\xx - F(\la,\xx)$ }
In this section we will study the differentiable properties of a
nonlinear operator $G$ induced by a discrete nonautonomous system
(\ref{main-system}) parametrized by a parameter space $S^1$ and
Fredholmness of an operator $D_xG$. We keep the notations and
assumptions from Section $2$. For any $\xx\in \lc$ define
\begin{align}\label{substitution}
\begin{split}
&F(\la,\xx)=(f_n(\la,x_n)),\;\;
F^{\infty}_{\pm}(\la,\xx)=(f^{\infty}_{\pm}(\la,x_n)),\\
&G(\la,\xx)=S\xx-F(\la,\xx),\;\;
G^{\infty}_{\pm}(\la,\xx)=S\xx-F^{\infty}_{\pm}(\la,\xx).
\end{split}
\end{align}

\begin{proposition}\label{proposition-estimation}
Under Assumptions $(A1)$--$(A4)$, $F(\la,\emph{$\xx$})$,
$F^{\infty}_{\pm}(\la,\emph{$\xx$})$, $G(\la,\emph{$\xx$})$ and
$G^{\infty}_{\pm}(\la,\emph{$\xx$})$ belong to \emph{$\lc$}, for
each $\la\in S^1$ and \emph{$\xx \in \lc$}.
\end{proposition}
\begin{proof} First of all  we will need the following lemma, which will be used repeatedly
in what follows. Working  on any  coordinate chart of $S^1$ we will denote with $\la$  also the coordinate of the point $\la\in S^1$.
\begin{lemma}\label{boundedness}
Assumptions $(A1)$--$(A2)$ imply that
\begin{equation*}
\sup_{(n,\la,y)\in \Z\times S^1\times \bar B_d(0,M)}
\left|\frac{\partial f_n(\la ,y)}{\partial x} \right|<\infty
\emph{ \;\;and \;\;}\sup_{(n,\mu,y)\in \Z\times S^1\times \bar
B_d(0,M)} \left|\frac{\partial f_n(\mu ,y)}{\partial \la}
\right|<\infty
\end{equation*}
for any $M>0$.
\end{lemma}
\begin{proof}
Let us observe that from  Assumption (A2) it follows easily that
\begin{equation}\label{Assumption-A2}
C_0:=\sup_{(n,\la)\in \Z\times S^1} \left|\frac{\partial
f_n(\la,0)} {\partial x}\right|<\infty.
\end{equation}
Fix $M>0$ and $\varepsilon>0$. Let $\delta>0$ be as in Assumption
$(A1)$. Take $(n,\la,y)\in \Z\times S^1\times \bar B_d(0,M)$. Then
there exists $n_0>0$ such that $n_0\leq M/\delta<n_0+1$.
Furthermore, there exist $0<k\leq n_0+1$ and points
$y_0=0,y_1,...,y_{k-1},y_k=y\in \bar B_d(0,M)$ such that
$|y_i-y_{i+1}|<\delta$, for $i=0,...,k-1$. Thus
\begin{align*}
&\left|\frac{\partial f_n(\la,y)}{\partial x}\right|\leq
\left|\frac{\partial f_n(\la,0)}{\partial
x}\right|+\left|\frac{\partial f_n(\la,y_1)}{\partial
x}-\frac{\partial f_n(\la,0)}{\partial x}\right|+
\left|\frac{\partial f_n(\la,y_2)}{\partial x}-\frac{\partial
f_n(\la,y_1)}{\partial x}\right|+\ldots+\\
&\left|\frac{\partial f_n(\la,y_{k-1})}{\partial x}-\frac{\partial
f_n(\la,y_{k-2})}{\partial x}\right|+\left|\frac{\partial
f_n(\la,y)}{\partial x}-\frac{\partial f_n(\la,y_{k-1})}{\partial
x}\right|\leq C_0+k\varepsilon\leq C_0+(n_0+1)\varepsilon,
\end{align*}
where $C_0$ is as in \eqref{Assumption-A2}. Observe that
$\displaystyle\frac{\partial f_n(\mu,0)}{\partial \la}=0$, for all
$\mu\in S^1$ and $n\in \Z$. It follows from the fact that
$f_n(\mu,0)=0$, for all $\mu\in S^1$ and $n\in\Z$. Consequently,
by the same reasoning as above, one can conclude the second part
of the assertion of the lemma.
\end{proof}
Now fix $\xx\in \lc$ and $\la\in S^1$. Let $M:=\sup\limits_{n\in
\Z}|\xx(n)|$. Then Lemma \ref{boundedness} implies that
\begin{equation*}
C:=\sup_{(n,y)\in \Z\times \bar B_d(0,M)}\left|\frac{\partial
f_n(\la,y)}{\partial x}\right|<\infty.
\end{equation*}
Hence, using the mean value theorem, we get
\begin{equation}\label{estimate}
|f_n(\la,y)|=|f_n(\la,y)-f_n(\la,0)|\leq\left(\sup_{s\in[0,1]}\left|\frac{\partial
f_n(\la,sy)}{\partial x}\right|\right)|y|\leq C|y|.
\end{equation}
Consequently,
\begin{align*}
\|F(\la,\xx)\|&=\|(f_n(\la,\xx(n)))\|\leq C\|(\xx(n))\|=C\|\xx\|,
\end{align*}
which implies that $F(\la,\xx)$ and hence also $G(\la,\xx)$ belong to $\lc.$

On the other hand, $|f_n(\la,y)|\xrightarrow[n\rightarrow\pm\infty]{}
|f^{\infty}_{\pm}(\la,y)|.$ Thus,  after passing to the limit in
(\ref{estimate}) as $n\ra \pm\infty$, we get
$ |f^{\infty}_{\pm}(\la,y)|\leq C|y|,$ for all $y\in \bar B_d(0,M)$.
From which it follows that  $\|F^{\infty}_{\pm}(\la,\xx)\| \leq C\|\xx\|.$
  Therefore, $F^{\infty}_{\pm}(\la,\xx)$, and $G^{\infty}_{\pm}(\la,\xx)$ belong to $\lc.$
\end{proof}

Using once again Lemma \ref{boundedness}, we  define
two families of linear bounded operators $T\colon S^1\times \lc\ra
\mathcal{L}(\lc)$ and $\tilde{T}\colon S^1\times \lc\ra
\mathcal{L}(\R,\lc)$ by
\begin{align}\label{frechet}
T(\la,\xx)\yy:=\left(\frac{\partial f_n(\la,x_n)}{\partial
x}y_n\right) \text{ and
\;\;}\tilde{T}(\la,\xx)z:=\left(\frac{\partial
f_n(\la,x_n)}{\partial \la}z\right)
\end{align}
for $\xx=(x_n), \yy=(y_n)\in \lc$, $\la\in S^1$ and $z\in\R$.
\begin{proposition}\label{diff-G}
The map  $F\colon S^1 \times \emph{$\lc$}\ra \emph{$\lc$}$ defined in \eqref{substitution}  is $C^1$.
 Moreover, $D_xF(\la,\emph{\xx})=T(\la,\emph{\xx})$
and $D_{\la}F(\la,\emph{\xx})=\tilde{T}(\la,\emph{\xx})$.
\end{proposition}
\begin{proof}
Observe that it suffices to prove that $D_xF$ and $D_{\la}F$ exist
and are continuous on $S^1\times \lc$. Firstly we prove that
$D_xF$ exists and $D_xF(\la,\xx)=T(\la,\xx)$. Fix $\xx\in \lc$ and
$\la\in S^1$. Then
\begin{align}\label{eqR}
\begin{split}
&R(\xx,{\bf h};\la):=\|F(\la,\xx+{\bf h})-F(\la,\xx)-T(\la,\xx){\bf
h}\|=\\&\left(\sum_{n\in
\Z}\left|f_n(\la,x_n+h_n)-f_n(\la,x_n)-\frac{\partial
f_n(\la,x_n)}{\partial x}h_n\right|^2\right)^{1/2},
\end{split}
\end{align}
where ${\bf h}\in \lc$ and $\la\in S^1$. We are to show that
$R(\xx,{\bf h};\la)\|{\bf h}\|^{-1}\ra 0$ as $\|{\bf h}\|\ra 0.$
Let
\begin{equation*}
c_n({\bf h};\la):=\sup_{s\in[0,1]}\left|\frac{\partial
f_n(\la,x_n+sh_n)}{\partial x}-\frac{\partial
f_n(\la,x_n)}{\partial x }\right|,
\end{equation*}
for $n\in\Z$. Then Assumption $(A1)$ implies that
\begin{equation*}
\sup_{n\in\Z}c_n({\bf h};\la)\xrightarrow[\|{\bf h}\|\rightarrow
0]{}0.
\end{equation*}
Then
\begin{align*}
&\left|f_n(\la,x_n+h_n)-f_n(\la,x_n)-\frac{\partial
f_n(\la,x_n)}{\partial x}h_n\right|=\\ &\left|\int_0^1
\frac{\partial f_n(\la,x_n+sh_n)}{\partial x}h_n ds-\frac{\partial
f_n(\la,x_n)}{\partial x}h_n\right|\leq \\ &|h_n|\int_0^1
\left|\frac{\partial f_n(\la,x_n+sh_n)}{\partial x
}-\frac{\partial f_n(\la,x_n)}{\partial x}\right|ds\leq \\ &
|h_n|\int_0^1c_n({\bf h};\la) ds\leq |h_n|\sup_{n\in\Z}c_n({\bf
h};\la).
\end{align*}
Hence, taking into account \eqref{eqR}, we infer that
\begin{equation}
0\leq R(\xx,{\bf h};\la)\leq\|{\bf h}\|\sup_{n\in\Z}c_n({\bf
h};\la),
\end{equation}
which implies that $R(\xx,{\bf
h};\la)\|{\bf h}\|^{-1}\ra 0$ as $\|{\bf h}\|\ra 0$. Now we will
show that $T\colon S^1\times \lc\ra \mathcal{L}(\lc)$ is
continuous. To this end, observe that
\begin{align}\label{cont-T}
\begin{split}
&\|(T(\la,\xx)-T(\mu,\yy))\zz\|^2=\sum_{n\in\Z}\left|\left(\frac{\partial
f_n(\la,\xx(n))}{\partial x}-\frac{\partial
f_n(\mu,\yy(n))}{\partial x}\right)\zz(n)\right|^2\\&
\leq\sum_{n\in\Z}\left|\frac{\partial f_n(\la,\xx(n))}{\partial
x}-\frac{\partial f_n(\mu,\yy(n))}{\partial x}\right|^2
|\zz(n)|^2.
\end{split}
\end{align}
Assumption $(A1)$ implies that for any $M>0$ and $\varepsilon>0$
there exists $\delta>0 $ such  for  all $(\la_1,x_1),(\la_2,x_2)
\in S^1\times \R^d$ with
$d\big((\la_1,x_1),(\la_2,x_2)\big)<\delta$, one has
\begin{equation*}
\sup_{n\in\Z}\left|\frac{
\partial f_n(\la_1,x_1)}{\partial x} -
\frac{\partial f_n(\la_2,x_2)}{\partial x}\right|<\varepsilon.
\end{equation*}
Fix $\xx\in \lc$ and $\varepsilon>0$ and take $\delta>0$ as above
(for $M:=2\|\xx\|$). Let
$d((\la,0),(\mu,0))<\min\{\delta/4,\|\xx\|\}$ and
$\|\xx-\yy\|<\min\{\delta/4,\|\xx\|\}$. Then for any $k\in\Z$ one
has $|\xx(k)-\yy(k)|\leq \|\xx-\yy\|$ and
\begin{align}\label{cont-T2}
\left|\frac{
\partial f_k(\la,\xx(k))}{\partial x} -
\frac{\partial f_k(\mu,\yy(k))}{\partial x}\right|\leq
\sup_{n\in\Z}\left|\frac{
\partial f_n(\la,\xx(k))}{\partial x} -
\frac{\partial f_n(\mu,\yy(k))} {\partial x}\right|<\varepsilon.
\end{align}
Thus, taking into account \eqref{cont-T} and \eqref{cont-T2}, we
infer that
\begin{align*}
\|(T(\la,\xx)-T(\mu,\yy))\zz\|^2\leq\sum_{n\in \Z}\varepsilon^2
|\zz(n)|^2=\varepsilon^2 \|\zz\|^2
\end{align*}
provided $d((\la,0),(\mu,0))<\min\{\delta/4,\|\xx\|\}$ and
$\|\xx-\yy\|<\min\{\delta/4,\|\xx\|\}$. Consequently, we deduce
that $T$ is continuous (with respect to the norm topology of
$\mathcal{L}(\lc)$).

Finally, it is not hard to see that the same reasoning as above
implies that $D_{\la}F(\la,\xx)=\tilde{T}(\la,\xx)$ and that
$D_{\la}F$ is continuous on $S^1\times \lc$. This completes the
proof.
\end{proof}

Now we will show that, for each $\la \in S^1,$  $G_\la$ is
Fredholm map of index $0$. Namely, $D_xG(\la,\xx)$ is a Fredholm
operator of index $0$ for all $(\la,\xx)$. For this purpose we
need to prove the following lemma.
\begin{lemma}\label{asymptotic-hyperbolic}
Under Assumptions $(A1)$--$(A2)$, for any $\emph{\xx}=(x_n)\in
\emph{\cc}(\R^d)$, one has
\begin{equation*}
\frac{\partial f_n(\la,x_n)}{\partial x}
\xrightarrow[n\rightarrow\pm\infty]{} a(\lambda,\pm\infty) \text{
$($uniformly with respect to $\la\in S^1)$.}
\end{equation*}
\end{lemma}
\begin{proof}
Fix $\xx\in\cc(\RN)$ and $\varepsilon>0$. Then Assumption $(A1)$
implies that there exists $\delta>0$ (for $M:=2\|\xx\|_{\infty}$)
such that
\begin{equation*}
\forall_{k\in\Z}\;\forall_{|y|\leq \delta}\; \forall_{\la\in S^1}
\left|\frac{
\partial f_k(\la ,y)}{\partial x} -
\frac{\partial f_k(\la,0)}{\partial x}\right|<\varepsilon.
\end{equation*}
Since $x_n \xrightarrow[n\rightarrow\pm\infty]{} 0$, it follows
that there exists $n_0>0$ such that $|x_n|\leq \delta$ for
$|n|\geq n_0$. Hence
\begin{equation*}
\forall_{|k|\geq n_0}\; \forall_{\la\in S^1} \left|\frac{
\partial f_k(\la,x_k)} {\partial x} -
\frac{\partial f_k(\la,0)}{\partial x}\right|<\varepsilon.
\end{equation*}
Now the assertion of lemma follows from Assumption $(A2)$.
\end{proof}

\begin{theorem}\label{fredholm-index-0}
Under Assumptions $(A1)$--$(A2)$, the map $G$ is $C^1.$ Moreover,
for any $\la \in S^1$ the map $G_\la \colon \emph{$\lc$}\ra
\emph{$\lc$}$ is  a Fredholm map of index $0$.
\end{theorem}
\begin{proof}
From Proposition \ref{diff-G} it follows directly that the map
$G(\la,\xx):=S\xx-F(\la,\xx)$ is $C^1$. Fix $\xx\in \lc$ and
$\la\in S^1$. Let $a_n(\la,x_n):=\displaystyle\frac{\partial
f_n(\la,x_n)}{\partial x}.$ From Proposition \ref{diff-G} it
follows that $D_xG(\la,\xx)$ is the operator $L_{\la}\colon \lc\ra
\lc$ defined by
\begin{equation}\label{linear}
L_{\la}\yy=(y_{n+1}-a_n(\la,x_n)y_n).
\end{equation}
Assumption $(A2)$ and Lemma \ref{asymptotic-hyperbolic} imply that
$\aa=(a_n(\la,x_n))$ is asymptotically hyperbolic. Consequently by
Theorem \ref{prop:ind}, the operator $L_\la$ is Fredholm with
index given by \eqref{dim}. Thus $\ind L_\la=0,$ since by $(A2)$
the stable subspaces at $\pm \infty$ have the same dimension.
\end{proof}

\begin{lemma}\label{equicontinuity-properness}
For any bounded sequence $(\emph{$\xx_n$})\subset \emph{$\lc$}$
the family of functions $\{G(\cdot,\emph{$\xx_n$})\colon S^1\ra
\emph{$\lc$}\}_{n\in \Z}$ is equicontinuous.
\end{lemma}
\begin{proof}
First observe that there exists $M>0$ such that $|\xx_n(k)|\leq
\|\xx_n\|\leq M$ for $n\in \N$ and $k\in\Z$.  From Lemmas
\ref{boundedness} and \ref{diff-G} it follows that
\begin{equation*}
L_M:=\sup_{(\mu,\xx)\in S^1\times\bar
B(0,M)}\|D_{\la}G(\mu,\xx)\|<\infty.
\end{equation*}
Integrating $D_{\la}G(\mu,\xx)$ over an arc of the circle joining $\la_1$ with $\la_2$  we get
\begin{align*}
&\|G(\la_2,\xx_n)-G(\la_1,\xx_n)\| \leq
L_M\,\text{dist}\,(\la_2,\la_1)
\end{align*}
which implies the equicontinuity of the family
$\left\{G(\cdot,\xx_n)\colon S^1\ra \lc\right\}_{n\in \Z}$.
\end{proof}

\section{Properness}
We are going to discuss   a properness criterion for the map $G$ adapting to our framework the approach used in \cite{Se-St}.

 \begin{definition}[\cite{Se-St}]\label{vanish}
We say that a sequence $(\xx_n)$ in $\lc$ vanishes uniformly at
infinity if, for all $\varepsilon>0$, there exist $n_0\in \N$ and
$m_0\in\N$ such that $|\xx_n(m)|\leq\varepsilon $ for all $n\geq
n_0$ and for all $|m|\geq m_0$
\end{definition}

\begin{lemma}\label{convergence-weak}
Let $\emph{\xx}\in \emph{$\lc$}$ and let $(\emph{\xx}_n)\subset
\emph{$\lc$}$. Then
$\emph{\xx}_n\xrightharpoonup[n\rightarrow\infty]{}\emph{\xx}$
weakly in $\emph{$\lc$}$ if and only if $(\emph{\xx}_n)$ is norm
bounded in $\emph{$\lc$}$ and
$p_k(\emph{\xx}_n)\xrightarrow[n\rightarrow\infty]{}
p_k(\emph{\xx})$, for all $k\in\Z$, where $p_k\colon
\emph{$\lc$}\ra \R$ are the canonical projections.
\end{lemma}
\begin{proof}
This is proved in \cite[Theorem 14.4]{bach-Nar}.
\end{proof}

\begin{lemma}\label{convergence-weak-II}
Let $(h_n)\subset \Z$ be a sequence such that
$\lim\limits_{n\ra\infty}|h_n|=\infty$ and let $\emph{\xx}\in
\emph{$\lc$}$. Define the sequence $(\widetilde{\emph{\xx}}_n)$ by
$\widetilde{\emph{\xx}}_n(m):=\emph{\xx}(m+h_n)$ for $m\in\Z$,
then
$\widetilde{\emph{\xx}}_n\xrightharpoonup[n\rightarrow\infty]{}
\emph{\oo}$ weakly.
\end{lemma}
\begin{proof}
It is a straightforward from Lemma \ref{convergence-weak}.
\end{proof}

\begin{lemma}\label{convergence}
Let \emph{$(\xx_n)$} be a bounded sequence in $\emph{$\lc$}$ and
let \emph{$\xx\in \lc$}. The following statements are equivalent:
\begin{enumerate}
\item[$(i)$] \emph{$\|\xx_n-\xx\|_{\infty}\xrightarrow[n\rightarrow\infty]{} 0$}.
\item[$(ii)$] \emph{$\xx_n\xrightharpoonup[n\rightarrow\infty]{}\xx$} in $\emph{$\lc$}$ and
\emph{$(\xx_n)$} vanishes uniformly at infinity.
\end{enumerate}
\end{lemma}
\begin{proof}
First, observe that the implication $(i)\Longrightarrow (ii)$ is
obvious. We are to show that $(ii)$ implies $(i)$. Fix
$\varepsilon>0$. Then there exist $m_0\in \N$ and $n_0\in\N$ such
that $|\xx_n(m)|<\varepsilon \text{ and } |\xx(m)|<\varepsilon$,
for all $|m|\geq m_0$ and $n\geq n_0$. Hence
$|\xx_n(m)-\xx(m)|<2\varepsilon$, for all $|m|\geq m_0$ and $n\geq
n_0$. Lemma \ref{convergence-weak} implies that there exists
$n_1\in \N$ such that $|\xx_n(m)-\xx(m)|<\varepsilon$, for all
$n\geq n_1$ and $|m|<m_0$. Thus we deduce that
$\|\xx_n-\xx\|_{\infty}\leq 2\varepsilon$, for all $n\geq
\max\{n_0,n_1\}$, which implies that
$\|\xx_n-\xx\|_{\infty}\xrightarrow[n\rightarrow\infty]{} 0.$
\end{proof}

The following lemma will play a crucial role in the proof the
properness.

\begin{lemma}[Shifted subsequence lemma]\label{shift}
Let \emph{$(\xx_n)\subset \lc$} be a bounded sequence. Then at
least one of the following properties must hold.
\begin{enumerate}
\item[$(i)$] \emph{$(\xx_n)$} vanishes uniformly at infinity.
\item[$(ii)$] There is a sequence $(l_k)\subset\Z$ with $\lim\limits_{k\ra
\infty}l_k=\infty$ and a subsequence \emph{$(\xx_{n_k})$} of
\emph{$(\xx_n)$} such that a sequence $(\widetilde{\emph{\xx}}_k)$
defined by \emph{$\widetilde{\xx}_k(m):=\xx_{n_k}(m+l_k)$}, for
$m\in \Z$, converges weakly in $\emph{$\lc$}$ to
\emph{$\widetilde{\xx}\neq \oo$}.
\item[$(iii)$] There is a sequence $(l_k)\subset\Z$ with $\lim\limits_{k\ra
\infty}l_k=-\infty$ and a subsequence \emph{$(\xx_{n_k})$} of
\emph{$(\xx_n)$} such that a sequence \emph{$(\widetilde{\xx}_k)$}
defined by \emph{$\widetilde{\xx}_k(m):=\xx_{n_k}(m+l_k)$}, for
$m\in \Z$, converges weakly to \emph{$\widetilde{\xx}\neq \oo$}.
\end{enumerate}
\end{lemma}
\begin{proof}
Assume that $(\xx_n)$ does not satisfy $(i)$. Then there exists
$\varepsilon>0$ such that for all $k\in\N$ there exists $m_k\in\Z$
with $|m_k|\geq k$ and there exists $n_k\geq k$ such that
$|\xx_{n_k}(m_k)|\geq\varepsilon.$ By passing to a subsequence, if
necessary, we may suppose that $(m_k)$ diverges either to $\infty$
or to $-\infty.$ Suppose that
$\lim\limits_{k\ra\infty}m_k=\infty$. Let $l_k\in\Z$ be defined by
$l_k:=m_k$. Let $\widetilde{\xx}_k:=(\xx_{n_k}(n+l_k))$. It is
clear that $\|\widetilde{\xx}_k\|=\|\xx_{n_k}\|.$ Since
$(\tilde{\xx}_k)$ is bounded and since  bounded subsets are weakly
compact, by passing to a subsequence if needed, we can assume that
$(\widetilde{\xx}_{k})$ converges weakly in $\lc$ to some element
$\widetilde{\xx}$. We will show that $\tilde{\xx}\neq \oo$.
Observe that $\varepsilon \leq|\widetilde{\xx}_{k}(0)|\leq K$.
Hence $\lim\limits_{k\ra\infty}
\widetilde{\xx}_{k}(0)=\widetilde{\xx}(0)\neq 0$, since
$\widetilde{\xx}_{k}\rightharpoonup \widetilde{\xx}$ in $\lc$ (see
Lemma \ref{convergence-weak}). The same reasoning shows that
$(iii)$ holds if $\lim\limits_{k\ra\infty}l_k=-\infty.$
\end{proof}

In the remaining part of this section we will study the properties
of the maps $F_\la(\xx) = F(\la,\xx)$ and  $G_\la
(\xx)=S\xx-F_\la(\xx)$ for a fixed value of parameter $\la\in
S^1$. We will consider our assumptions $(A1)$--$(A4)$ to hold for
the constant family $\ff(\la,\xx) = \ff(\xx)$  and drop $\la $
everywhere from the notations. For example, the derivative of $G$
with respect to the second variable will be denoted by $DG(\xx)$
instead of $D_xG(\la,\xx)$).

\begin{lemma}\label{weakly-continuous}
$F\colon \emph{$\lc$}\ra \emph{$\lc$}$,
$F^{\infty}_{\pm}\colon\emph{$\lc$}\ra \emph{$\lc$}$, $G
\colon\emph{$\lc$}\ra \emph{$\lc$}$ and $G^{\infty}_{\pm}
\colon\emph{$\lc$}\ra \emph{$\lc$}$ are weakly continuous.
\end{lemma}
\begin{proof}
Fix $\xx\in \lc$. Let $\xx_k
\xrightharpoonup[k\rightarrow\infty]{}\xx$. We will show that
$F^{\infty}_{\pm}(\xx_k)\xrightharpoonup[k\rightarrow\infty]{}
F^{\infty}_{\pm}(\xx)$. To this end, by Lemma
\ref{convergence-weak} it suffices to show that
$(F^{\infty}_{\pm}(\xx_k))$ is norm bounded in $\lc$ and
$p_n(F^{\infty}_{\pm}(\xx_k))\xrightarrow[k\rightarrow\infty]{}
p_n(F^{\infty}_{\pm}(\xx))$, for all $n\in\Z$, where $p_n\colon
\lc\ra \R$ are the canonical projections. First observe that there
exists $M>0$ such that $\|\xx_k\|<M$ for all $k\in \N$ and hence
$|\xx_k(n)|<M$ for $k\in \N$ and $n\in\Z$. From Lemma
\ref{boundedness} we infer that
\begin{equation*}
C:=\sup_{(n,y)\in \Z\times \bar B_d(0,M)}\left|
Df_n(y)\right|<\infty.
\end{equation*}
Thus, reasoning as in the proof of Proposition
\ref{proposition-estimation}, we get
$\|F^{\infty}_{\pm}(\xx_k)\|\leq C\|\xx_k\|<CM,$ for all $k\in\N$.
On the other hand, since $f_n \xrightarrow[n\rightarrow\pm\infty]{} f^{\infty}_{\pm},$
uniformly on bounded subsets of $\R^d$, it follows that the map
$f^{\infty}_{\pm}\colon \R^d\ra\R^d$ is continuous.

 Since $\xx_k(n)\xrightarrow[k\rightarrow\infty]{}\xx(n)$, we deduce that
\begin{equation*}
p_n(F^{\infty}_{\pm}(\xx_k))=f^{\infty}_{\pm}(\xx_k(n))\xrightarrow[k\rightarrow\infty]{}f^{\infty}_{\pm}(\xx(n))=
p_n(F^{\infty}_{\pm}(\xx)),
\end{equation*}
which completes the proof that $F^{\infty}_{\pm}$ is weakly
continuous. The same proof shows that  $F$
is also weakly continuous which,  on its turn implies that both $G$ and $G^{\infty}_{\pm}$ are weakly continuous.  This completes the proof.
\end{proof}

\begin{lemma}\label{properness}
For the Fredholm map $G= S-F$ the following statements are
equivalent:
\begin{enumerate}
\item[$(a)$] The restricted map $G_{|D}$ is proper for each closed and
bounded subset $D$ of \emph{$\lc$}.
\item[$(b)$] If \emph{$(\xx_n)$} is a bounded sequence in
\emph{$\lc$} such that $(G(\emph{$\xx_n$}))$ is convergent in
\emph{$\lc$}, then \emph{$(\xx_n)$} has a convergent subsequence
in ${\bf c}(\R^d)$.
\end{enumerate}
\end{lemma}
\begin{proof}
A map $G$ is proper on closed bounded subsets if and only if any
bounded sequence $(\xx_n)$ such that  $G(\xx_n) $ is convergent
has a subsequence convergent to some point of the set. Hence, that
$(a)$ implies $(b)$ follows plainly from the continuity of the
embedding $\lc\hookrightarrow {\bf c}(\RN).$ In order to show that
$(b)$ implies $(a), $ let $(\xx_n)$ be a bounded sequence such
that
\begin{equation}\label{F-I}
\|G(\xx_n)-\yy\|\xrightarrow[n\rightarrow\infty]{}0,
\end{equation}
where $\yy\in \lc$. Since the ball  $\bar B(\oo,C)$ in $\lc$ is
weakly-compact we can assume that $\xx_n\rightharpoonup \xx$ in
$\lc$. By  $(b)$,  passing to a subsequence if necessary, we can
assume  that $\|\xx_n-\xx\|_{\infty}\ra 0$ as $n\ra \infty$.
Since, by Lemma \ref{weakly-continuous}, $G$ is weakly continuous,
we have $G(\xx_n)\xrightharpoonup[n\rightarrow\infty]{} G(\xx)$,
and consequently $\yy=G(\xx).$

\begin{claim}\label{propernessdiff}
For the above sequence  one has
\begin{equation}\label{GDG}
\|G(\emph{$\xx_n$})-G(\emph{$\xx$})-D
G(\emph{$\xx$})(\emph{$\xx_n$}-\emph{$\xx$})\|\xrightarrow[n\rightarrow\infty]{}0.
\end{equation}
\end{claim}
\begin{proof}
The assertion \eqref{GDG}  is equivalent to \begin{equation}\label{F-II}
\|F(\xx_n)-F(\xx)-DF(\xx)(\xx_n-\xx)\|\xrightarrow[n\rightarrow\infty]{}0.
\end{equation}

Fix $\varepsilon>0$. Then there exists
$\delta>0$ such that if $\|\tilde{\xx}-\xx\|_{\infty}<\delta$,
then
\begin{equation}\label{sup-seq}
\sup_{k\in\Z}\left|Df_k(\tilde{\xx}(k))-Df_k(\xx(k))\right|<\varepsilon
\text{ (see Assumption (A1))}.
\end{equation}
Let  $n_0>0$ be such that $\|\xx_n-\xx\|_{\infty}<\delta$,
for all $n\geq n_0$. Fix $n\geq n_0 $ and $k\in\Z$. Then
\begin{align*}
&f_k(\xx_n(k))-f_k(\xx(k))-Df_k(\xx(k))(\xx_n(k)-\xx(k))=\\
&\int_0^1
Df_k\big(\xx_n(k)-s[\xx_n(k)-\xx(k)]\big)\big(\xx_n(k)-\xx(k)\big)ds-\int_0^1
Df_k(\xx(k))(\xx_n(k)-\xx(k))ds=\\
&\int_0^1\Big(Df_k\big(\xx_n(k)-s[\xx_n(k)-\xx(k)]\big)-
Df_k(\xx(k))\Big)\left(\xx_n(k)-\xx(k)\right)ds.
\end{align*}
Hence
\begin{align*}
&\left|f_k(\xx_n(k))-f_k(\xx(k))- Df_k(\xx(k))(\xx_n(k)-\xx(k))\right|=\\
&\left|\int_0^1\Big(Df_k\big(\xx_n(k)-s[\xx_n(k)-\xx(k)]\big)-
Df_k(\xx(k))\Big)\big(\xx_n(k)-\xx(k)\big)ds\right|\leq\\
&\int_0^1\left|D f_k\big(\xx_n(k)-s[\xx_n(k)-\xx(k)]\big)-
Df_k(\xx(k))\right||\xx_n(k)-\xx(k)|ds.
\end{align*}
Taking into account \eqref{sup-seq}, we obtain
\begin{align*}
&\int_0^1\left|D f_k\big(\xx_n(k)-s[\xx_n(k)-\xx(k)]\big)-
Df_k(\xx(k))\right||\xx_n(k)-\xx(k)|ds\\&\leq\int_0^1\varepsilon
\left|\xx_n(k)-\xx(k)\right|ds=\varepsilon
\left|\xx_n(k)-\xx(k)\right|.
\end{align*}
Thus,  we arrive at
\begin{align*}
&\|F(\xx_n)-F(\xx)-DF(\xx)(\xx_n-\xx)\|^2\\&=
\sum_{k\in\Z}\left|f_k(\xx_n(k))-f_k(\xx(k))-
Df_k(\xx(k))(\xx_n(k)-\xx(k))\right|^2\\
&\leq\sum_{k\in\Z}\varepsilon^2
\left|\xx_n(k)-\xx(k)\right|^2=\varepsilon^2\|\xx_n-\xx\|^2\leq
\varepsilon^2 (2C)^2,
\end{align*}
for $n\geq n_0$. This proves that
\begin{equation}\label{F-IIa}
\|F(\xx_n)-F(\xx)-DF(\xx)(\xx_n-\xx)\|\xrightarrow[n\rightarrow\infty]{}0
\end{equation}
and the claim.
\end{proof}

By the  above claim   and  \eqref{F-I} we have $
\|DG(\xx)(\xx_n-\xx)\|\xrightarrow[n\rightarrow\infty]{}0. $ By
Riesz criterion, Fredholm operators  are invertible modulo compact
operators. Therefore,  there exist a  bounded operator $B\colon
\lc\ra \lc$ and a compact  operator $K\colon \lc\ra \lc$ such that
$B\circ DG(\xx)=I+K$. In turn this implies that
\begin{align}
\begin{split}
&\|\xx_n-\xx\|=\|(B\circ DG(\xx)-K)(\xx_n-\xx)\|\leq\|\big(B\circ
DG(\xx)\big)(\xx_n-\xx)\|+\|K(\xx_n-\xx)\|\leq\\
&\|B\|\|DG(\xx)(\xx_n-\xx)\|+\|K(\xx_n-\xx)\|.
\end{split}
\end{align}
Since a compact operator $K\colon \lc\ra \lc$ maps weakly
convergent sequences onto norm convergent sequences, we infer that
\begin{equation}\label{F-III}
\|K(\xx_n-\xx)\|\xrightarrow[n\rightarrow\infty]{}0.
\end{equation}
Thus, in view of \eqref{F-IIa}--\eqref{F-III}, one obtains
$\|\xx_n-\xx\|\xrightarrow[n\rightarrow\infty]{}0$, which
completes the proof.
\end{proof}
Given $m\in\Z$, by $S_m\colon \lc\ra \lc$ we will denote the
m-shift operator defined by
\begin{equation}\label{shift-S-m}
S_m\xx:=(x_{n+m}).
\end{equation}

\begin{lemma}\label{estimate-II}
For any $\emph{$\xx$}\in \emph{$\lc$}$ and $m\in\Z$, $p_m\left(S_k
G(\emph{$\xx$})-S_k G^\infty_{\pm}(\emph{$\xx$})\right)
\xrightarrow[k\rightarrow\pm\infty]{}0$ $($uniformly on  any
bounded set $B\subset \emph{$\lc$})$.
\end{lemma}
\begin{proof}
Let $B\subset \lc$ be a bounded subset. Then there exists a
constant $C$ such that $|\xx(n)|\leq C$, for all $n\in\Z$ and
$\xx\in B.$ Assumption $(A4)$ implies that there exists a positive
integer $\tilde{n}_0=n(\varepsilon,B)$ such that $|f_{\pm
k}(x)-f^{\infty}_{\pm}(x)|<\varepsilon$, for $k\geq \tilde{n}_0$
and $x$ with $|x|\leq C$. Consequently,
\begin{equation}\label{nierownosc}
|f_{\pm k}(\xx(\pm k))-f^{\infty}_{\pm}(\xx(\pm k))|<\varepsilon,
\end{equation}
for $k\geq \tilde{n}_0$ and $\xx\in B.$ Finally, taking into
account \eqref{nierownosc}, we deduce that for any $k\geq
n_0:=\tilde{n}_0+|m|$ one has
\begin{equation*}
|p_{m\pm k}(G(\xx)-G^{\infty}_{\pm}(\xx))|=|p_{m\pm k}(
F(\xx)-F^{\infty}_{\pm}(\xx))| =|f_{m\pm k}(\xx(m\pm
k))-f^{\infty}_{\pm}(\xx(m\pm k))|<\varepsilon,
\end{equation*}
for all $\xx\in B$. Finally, it suffices to observe that
\begin{equation*}
\left|p_m\left(S_l G(\xx)-S_l
G^\infty_{\pm}(\xx)\right)\right|=|p_{m+l}(G(\xx)-G^{\infty}_{\pm}(\xx))|.
\end{equation*}
This completes the proof.
\end{proof}
\begin{theorem}\label{Theorem-properness} Under Assumptions $(A1)$--$(A2)$ and $(A4)$,
$G\colon \emph{$\lc$}\ra \emph{$\lc$}$ is proper on closed bounded
subsets of $\emph{$\lc$}$.
\end{theorem}
\begin{proof}
In view of Lemma \ref{convergence} and Lemma \ref{properness} it
suffices to show that any bounded sequence $(\xx_n)$ in $\lc$ such
that $\|G(\xx_n)-\yy\|\xrightarrow[n\rightarrow\infty]{} 0$ for
some $\yy\in \lc$ has a weakly convergent subsequence which
vanishes uniformly at infinity. Since $(\xx_n)$ is bounded, we may
assume without loss of generality that
$\xx_n\xrightharpoonup[n\rightarrow\infty]{} \xx$ weakly in $\lc$
for some $\xx\in \lc$.  If the alternative $(ii)$ of Lemma \ref{shift} holds, $(\xx_n)$ has a subsequence
$x_{n_k}$ whose translates $\widetilde{\xx}_k(n):=\xx_{n_k}(n+l_k)$ converge weakly to   $\widetilde{\xx}\neq \oo.$

Observe that
$\|G(\xx_n)-\yy\|=\|S_{l_k}G(\xx_n)-\widetilde{\yy}_k\|,$
where $\widetilde{\yy}_k:=S_{l_k}\yy,$ and therefore
\begin{align*}
&\|S_{l_k}G(\xx_{n_k})-\widetilde{\yy}_k\|\leq
\|S_{l_k}G(\xx_{n_k})-S_{l_k}G(\xx_{n})\|+\|S_{l_k}G(\xx_{n})-\widetilde{\yy}_k\|=\\&
\|G(\xx_{n_k})-G(\xx_{n})\|+\|G(\xx_{n})-\yy\|,
\end{align*}
which shows that
\begin{equation}\label{prop1}
\|S_{l_k}G(\xx_{n_k})-\widetilde{\yy}_k\|\xrightarrow[k\rightarrow\infty]{}0.
\end{equation}
Now let us fix $m\in\Z$. By Lemma \ref{estimate-II}
\begin{equation}\label{prop2}
\left|p_m\left(S_{l_{k}}G(\xx_{n_k})-S_{l_{k}}
G_{+}^{\infty}(\xx_{n_k})\right)\right|
\xrightarrow[k\rightarrow\infty]{}0
\end{equation}
and hence $ \left|
p_m(S_{l_{k}}G^{\infty}_{+}(\xx_{n_k})-\widetilde{\yy}_{k})
\right|\xrightarrow[k\rightarrow\infty]{}0$ as well. Since
\begin{align*}
S_{l_{k}}(G^{\infty}_{+}(\xx_{n_{k}}))=S(\xx_{n_{k}}
(n+l_{k}))-(f^{\infty}_{+}(\xx_{n_{k}}(n+l_{k})))
=S(\widetilde{\xx}_{k}(n))-(f^{\infty}_{+}(\widetilde{\xx}_{k}(n)))
=G^{\infty}_{+}(\widetilde{\xx}_{k}),
\end{align*}
we deduce that
\begin{equation*}
|p_m(G^{\infty}_{+}(\widetilde{\xx}_{k})-\widetilde{\yy}_{k}
)|\xrightarrow[k\rightarrow\infty]{}0.
\end{equation*}

Since the sequence
$(G^{\infty}_{+}(\tilde{\xx}_{k})-\widetilde{\yy}_{k})$ is bounded
in $\lc,$ it follows from Lemma \ref{convergence-weak} that
$G^{\infty}_{+}(\tilde{\xx}_{k})-\widetilde{\yy}_{k}\xrightharpoonup[k\rightarrow\infty]{}
\oo$ in $\lc.$ But Lemma \ref{convergence-weak-II} implies that
$\widetilde{\yy}_{k}\xrightharpoonup[k\rightarrow\infty]{}\oo$ in
$\lc$. Hence we get that
$G^{\infty}_{+}(\widetilde{\xx}_{k})\xrightharpoonup[k\rightarrow\infty]{}
\oo$ in $\lc$. However, the weak sequential continuity of
$G^{\infty}_{+}\colon \lc\rightarrow \lc$ implies that $
G^{\infty}_{+}(\widetilde{\xx}_{k})\xrightharpoonup[k\rightarrow\infty]{}
G^{\infty}_{+}(\widetilde{\xx}) $ weakly in $\lc$, which implies
that $G^{\infty}_{+}(\widetilde{\xx})=\oo,$ contradicting
Assumption $(A4)$. This shows that the sequence $(\xx_{n})$ cannot
have the property $(ii$) of Lemma \ref{shift}. By a similar
arguments we can exclude the property $(iii)$ in Lemma
\ref{shift}. This completes the proof.
\end{proof}
\section{Proof of the main theorem}
For the proof    we will use an extension of Leray-Schauder degree
to proper Fredholm maps of index $0$ introduced in  \cite{Pej-Rab}
under the name of {\it base point degree.}  What is of interest
for us is a very special form of the homotopy principle  for this
degree. As a consequence of Kuiper's theorem about the
contractibility of the linear group of a Hilbert space, only the
absolute value of  any  degree theory for general Fredholm maps
extending the Leray-Schauder degree can be homotopy invariant. The
most interesting characteristic  of  the base point degree
consists   in that the change in sign of the degree along a
homotopy can be determined using an invariant of paths of linear
Fredholm operators of index zero called  {\it parity.}

The parity is defined as follows: Let  $L\colon [a,b]\ra
\Phi_0(X,Y)$ be a path of Fredholm operators such that both $L_a$
and $L_b$ are invertible.  It can be shown \cite{Fi-Pej-88}  that
there exists a path of invertible operators $P\colon [a,b]\ra\,
GL(Y,X)$ such that $ L_t P_t = \Id_Y-K_t$, where  $K_t$ is a
family of operators with $\im  K_t$ contained in a finite
dimensional subspace $V$ of $Y.$ Such a path $P$ is called a
(regular) parametrix. If $P$ is a parametrix, then $ L_a P_a$ and
$ L_b P_ b$ are invertible, and so are their restrictions $C_a,C_b
\colon V \ra V$ to the subspace $V$ containing the images of
$K_a,K_b.$ The {\sl parity} of the path $L$ is the element
$\sigma(L) \in \Z_{2}= \{1,-1\}$ defined  by
\begin{equation*}
\sigma(L) = \sign\det C(a)\sign \det C(b).
\end{equation*}
The above definition is independent of the choices involved. The
parity is multiplicative and invariant under homotopies of paths
with invertible end points. If the path $L$ is closed, i.e.,
$L_a=L_b,$ then, via the identification $S^1\simeq [a,b]/\{a,b\}$
we can consider the path $L $ as a map $L \colon S^1\rightarrow
\Phi_{0} (X,Y)$  and relate  the parity of a closed path with the
obstruction to triviality  $w_1\colon \widetilde{KO}(S^1) \ra
\Z_2.$

\begin{lemma}[\cite{Fi-Pej-88}, Proposition  1.6.4 or \cite{Pej-Ski}, Proposition 3.1]\label{parity-lemma}
Under the above assumptions,
\begin{equation}\label{paritywhitney}
\sigma(L) = w_1(\Ind L).
\end{equation}
\end{lemma}
Now let us recall the construction of the degree. A
$C\ö1$-Fredholm map of index $0$ is by definition a $C^1$-map
$f\colon\mathcal{O}\to Y$  such that the  Fr\'{e}chet derivative
$Df(x)$ of $f$ at $x$ is a Fredholm operator of index $0.$
\newline\indent Let $\mathcal{O}\subset X$ be an open simply connected set.
An {\it admissible triple}  $(f,\Omega,y)$   is  defined by  a
$C\ö1$-Fredholm map of index $0,$ $f\colon \mathcal{O}\ra Y$ which
is proper on closed bounded subsets of $\cO,$  an open bounded set
$\Omega $ whose closure is contained in $\mathcal{O}$ and a point
$y\in Y$  such that $y\not\in f(\partial\Omega).$ The construction
of  \cite{Pej-Rab} associates to each admissible triple
$(f,\Omega,y)$ and each point $b \in \cO$, called {\it base
point}, an integral number $\text{deg}_b(f,\Omega,y)\in \Z$
called {\it base point degree}. A {\it regular base point} is a
point $b\in \cO$ which is a regular point  of the map $f$ (i.e.,
$Df(b)$ is an isomorphism). If $b$ is a regular base point and $y$
is a regular value of the restriction of $f$ to $\Omega$, then the
base point degree is defined by
\begin{equation}\label{deg}
\text{deg}_b(f,\Omega,0)=\sum_{x\in
f^{-1}(0)}\sigma(Df\circ\gamma),
\end{equation}
where $\gamma$ is any path in $\mathcal{O}$ joining $b$ to $x.$
\newline\indent In order to define the degree for any $y\in Y$ it is used an
approximation result by regular values (see \cite{Pej-Rab}  for
details).  If  $b $ is a singular point of $f$, then by definition
$\text{deg}_b(f,\Omega,y)=0.$ The degree such defined  has the
usual additivity, excision and normalization properties. For
$C\ö1$-maps that are compact perturbations of the identity  it
coincides with the Leray-Schauder degree.  However the homotopy
property requires a different formulation (for the sake of
definiteness  we will take $y=0$):
\begin{lemma}\label{varhomotopy}
Let $h\colon [0,1]\times\mathcal{O}\ra Y$  be a continuous map
that is proper on closed bounded subsets and such that each $h_t$
is a $C^1$-Fredholm map.  Let   $\Omega$ be  an open bounded
subset of $\mathcal{O}$ such that $0\not\in
h([0,1]\times\partial\Omega).$ If $b_i\in \mathcal O$  is a
regular base point for  $h_i:=h(i,\cdot),i= 0,1,$ then
\begin{equation*}
\emph{deg}_{b_0}(h_0,\Omega,0)=\sigma(M)\emph{deg}_{b_1}
(h_1,\Omega,0),
\end{equation*}
where $M\colon [0,1] \ra \Phi_0(X,Y)$ is the path $ L\circ
\gamma,$ where $L(t,x):=D_xh(t,x)$ and $\gamma$ is any path
joining $(0,b_0)$ to $(1,b_1)$ in $[0,1]\times \mathcal{O}.$
\end{lemma}
Notice that $\sigma(M)$ is independent of the choice of the path
$\gamma$ because $[0,1]\times\cO$ is simply connected. The proof
of the lemma \ref{varhomotopy}  can be found in (\cite[Lemma
2.3.1]{Pejs-1}).  Here we will need a minor generalization of the
above property.
\begin{lemma}[Generalized Homotopy Property]\label{generalized-homotopy}
Let $h\colon [0,1]\times\mathcal{O}\ra Y$  be a continuous map
that is proper on closed bounded subsets and such that each $h_t $
is a $C^1$-Fredholm map. Let $\Omega$ be an open and  bounded set
whose closure is contained in $[0,1]\times \mathcal{O}$ such that
$0\not\in h(\partial \Omega).$ If $b_i\in \mathcal O$  is a
regular base point for  $h_i:=h(i,\cdot),i=0,1,$ then
\begin{equation}\label{homotopy}
\emph{deg}_{b_0}(h_0,\Omega_0,0)=\sigma(M)\emph{deg}_{b_1}(h_1,\Omega_1,0),
\end{equation}
where $M$ is as above and $\Omega_t:=\{x\in X\mid (t,x)\in
\Omega\}$, for $t\in[0,1].$
\end{lemma}

\begin{proof}
We will prove the lemma assuming that
$\text{deg}_{b_0}(h_0,\Omega_0,0)\neq 0,$ which is the only case
that we will need in the sequel and leave  to the reader the
completion of the proof in the general case. Since  the degree of
a map without regular points vanishes,  being the absolute value
of the degree invariant under homotopy, it follows from our
assumption that for all $\tau \in [0,1]$ there exists a regular
base point for $h_{\tau}$. Let $C(t):=\{x\in X\mid h(t,x)=0 \}
\cap \Omega=\{x\in X\mid h(t,x)=0 \} \cap \bar\Omega.$ Now we will
prove that the map $[0,1]\ni t\mapsto C(t)\subset Y$ is upper
semicontinuous, i.e., for any point $t_0\in [0,1]$ and any open
neighborhood $V\subset X$ such that $C(t_0)\subset V$ there exists
an open neighborhood $U_{t_0}$ of $t_0$ in $[0,1]$ such that
$C(t)\subset V$ for all $t\in U_{t_0}$. Indeed, let $t_0$ and $V$
be as above. Assume on the contrary that for any $\varepsilon>0$
there exists $t_{\varepsilon}\in
(t_0-\varepsilon,t_0+\varepsilon)\cap[0,1]$ such that
$C(t_{\varepsilon})\cap (X\setminus V)\neq\emptyset$. Thus there
exists a sequence $(t_n,x_n)\in \bar\Omega$ such that
$t_n\xrightarrow[n\rightarrow\infty]{} t_0$, $h(t_n,x_n)=0$ and
$x_n\in X\setminus V$. Since $h^{-1}(0)\cap \bar\Omega$ is
compact, we can assume that $x_n\xrightarrow[n\rightarrow\infty]{}
x_0$ for some $x_0\in X\setminus V$. Furthermore, the continuity
of $h$ implies that $h(t_0,x_0)=0$. Since $(t_0,x_0)\in
\bar\Omega$, it follows that $x_0\in C(t_0)\subset V$, which
contradicts the fact that $x_0\in X\setminus V$.
\newline\indent
Thus, given any  point $t\in I=[0,1]$ and an open neighborhood
$V_t\subset \Omega_t $ of $C(t)$ we can find an  open neighborhood
$U_t$ of $t$ in $I$ such that $C(t') \subset V_t $ for all $t'\in
U_t.$ Let $\delta>0$ be the Lebesgue number of the covering $\{U_t
\mid t\in I\}$  and let $0=t_0 <t_1< \dots <t_n=1$ be a partition
of $I$ with mesh less than $\delta.$ Then any subinterval $I_i=
[t_{i-1}, t_i]$ is contained in some element $U_{\tau_i}$ of the
covering and therefore $C(I_i ) \subset V_{\tau_i}.$ This means
that  the graph of   $C_{\mid I_i},$ which, by its very
definition, is the set $\{ (t,x)\in  \Omega \mid t\in I_i,
h(t,x)=0 \},$   must be contained in  $I_i \times V_{\tau_i}.$
\newline\indent
Since $0\notin h(\bar\Omega\cap(I_i\times X)\setminus(I_i \times
V_{\tau_i})),$ choosing any  regular base point  $b_i$ of
$h_{t_i},$ we can apply Lemma \ref{varhomotopy} to the map
$h\colon I_i \times \bar  V_{\tau_i} \ra Y$     and use the
excision property of the degree in order to obtain
\begin{equation*} \label{partition}
\text{deg}_{b_{i-1}}(h_{t_{i-1}},\Omega_{t_{i-1}},0)=
\text{deg}_{b_{i-1}}(h_{t_{i-1}},V_{\tau_{i}},0)= \sigma(M_i)
\text{deg}_{b_{i}}(h_{t_{i}},V_{\tau_{i}},0) = \sigma(M_i)
\text{deg}_{b_i} (h_{t_i},\Omega_{t_i},0),
\end{equation*}
where $M_i (t) =D_x h(\gamma_i(t))$   and $\gamma_i \colon I_i \ra
[0,1]\times \mathcal{O}$ is any path joining $(t_{i-1},b_{i-1})$
to $(t_i,b_i).$\newline Now \eqref{homotopy} follows from the
above identities, because, by the  multiplicative property of the
parity, $\prod_{i=1}^n \sigma(M_i) =\sigma(M),$ where $M (t) =D_x
h(\gamma(t))$ and $\gamma$ is the concatenation of all paths
$\gamma_i, \, 1\le i\leq n.$  This completes the proof.
\end{proof}
We will need also the following  result. Recall that nonempty
subsets $A$, $B$ of a space $X$ are separated (in $X$) if there
exists  open (and hence closed)  neighborhoods $U_A\supset A$,
$U_B\supset B$ in $X$ such that $U_A\cap U_B=\emptyset$ and
$U_A\cup U_B=X.$ Two sets are connected (to each other) in $X$ if
there is a connected set $Y\subset X$ with $A\cap Y\neq\emptyset$
and $B\cap Y\neq\emptyset.$  Whyburn's Lemma (see \cite{Al}) says that  if $A$, $B$ are
closed subsets of a compact space $X$ that are not connected to
each other, then they are separated in $X$.

Now let us prove our main Theorem \ref{Theorem-A}.
\begin{proof} We first prove \emph{$[i]$.}
Let $X=\lr$ and let $\cS =G^{-1}(\oo) \subset S^1\times X.$ $\cS$
is a locally compact space and in fact $\sigma$-compact, since
 $\cS = \bigcup\limits_{k\in \N} \cS \cap \left(S^1\times \bar B(\oo,k)\right).$ Let $\cC_0$ be
the connected component  of $ \cT_0$ in $\cS$. Suppose that
$\cC_0$ is bounded and let $ W$ be any bounded closed neighborhood
of $\cC_0.$  Since  $\cC_0$ is a maximal connected set,  $A:=\cS
\cap
\partial W$ is not connected with $\cC_0$ in the compact space
$\cS\cap W.$ Therefore there exist two compact subsets  $K_0, K_1$
of $\cS\cap W$ separating the component $\mathcal{C}_0$ from $A.$
Let $d=\text{dist}(K_0,K_1)>0,$  and let $\Omega:=\{(\la,\xx)\in
S^1\times X\mid d((\la,\xx),K_0)<1/2d\}.$ Then $\Omega$ is an open
bounded neighborhood of $\cC_0$ in $S^1\times X$ such that
$G(\la,\xx) \neq \oo$ on $\partial\Omega.$ For simplicity, we can
assume that $\la_0$ satisfying Assumption $(A3)$ equals $1$. Let
$q \colon [0,1]\ra S^1$ be the identification map taking $0,1$
into $1\in S^1.$ Let us consider the homotopy   $H\colon [0,1]
\times X\ra X$ defined by $H(t,x)=G(q(t),x).$  Clearly $H$ is a
continuous family of $C\ö1$-Fredholm maps. Put
$\Omega':=p^{-1}(\Omega).$  By construction $\Omega'$ is an open
bounded subset of $[0,1]\times X$ and $H$ has no zeros on the
boundary of $\Omega'.$ We will apply  the generalized homotopy
principle to $H$ on $\Omega'$  in order  to obtain a
contradiction. For this we need to  show:

\begin{lemma}\label{proper-G}
The restriction of $H$ to any closed bounded subset of $[0,1]
\times X$  is proper.
\end{lemma}
\begin{proof}
Let $K$ be a compact subset of $X $ and let $D$ be a closed
bounded subset of $[0,1] \times X.$ We  have to  show that
$(H_{\mid D})^{-1}(K)$ is compact. To this end, take any sequence
$(t_n,\xx_n)\in(H_{\mid D})^{-1}(K).$ Without loss of generality
we can assume that there exist $t_0\in S^1$ and $\yy\in X$ such
that
\begin{equation*}
t_n\xrightarrow[n\rightarrow\infty]{}t_0\in [0,1] \text{ and }
\|H(t_n,\xx_n)-\yy\|\xrightarrow[n\rightarrow\infty]{}0.
\end{equation*}
Since, by Lemma \ref{equicontinuity-properness}, the family of functions $\{G(\cdot,\xx_n)\colon S^1\ra
\lc\}_{n\in \Z}$ is equicontinuous and $H(\cdot,\xx_n)=G(q(\cdot),\xx_n)$, we infer that the family $\{H(\cdot,\xx_n)\colon S^1\ra \lc\}_{n\in \Z}$ is also equicontinuous. Now we will show that
\begin{equation}\label{52}
\|H(t_0,\xx_n)-\yy\|\xrightarrow[n\rightarrow\infty]{}0.
\end{equation}
For this purpose, fix $\varepsilon>0$. Then there exists $k_0>0$ such that
\begin{align*}
&\|H(t_m,\xx_n)-H(t_0,\xx_n)\|<\varepsilon/2 \text{ for }m\geq k_0 \text{ and for all }n\in \N,\\
&\|H(t_n,\xx_n)-y\|<\varepsilon/2 \text{ for }n\geq k_0.
\end{align*}
Hence
\begin{equation*}
\|H(t_0,\xx_n)-\yy\|\leq \|H(t_n,\xx_n)-\yy\|+\|H(t_n\xx_n)-H(t_0,\xx_n)\|<\varepsilon/2+\varepsilon/2=\varepsilon,
\end{equation*}
for $n\geq k_0$, which proves \eqref{52}.

From Theorem \ref{Theorem-properness} it follows that
$H_{t_0}\colon X\ra X$ is proper on closed and bounded subsets of
$X$ and therefore there exists a subsequence $(\xx_{n_k})$ of
$(\xx_n)$ and $\xx\in X$ such that
$\|\xx_{n_k}-\xx\|\xrightarrow[k\rightarrow\infty]{}0$. Therefore,
$(t_{n_k},\xx_{n_k})\xrightarrow[k\rightarrow\infty]{} (t_0,\xx)$
in $D.$ Thus we conclude that $(H_{\mid D})^{-1}(K)$ is compact,
which completes the proof of lemma.
\end{proof}

By the above lemma  $H$ is an admissible homotopy  with
$H_0:=H(0,\cdot)=H(1,\cdot)=:H_1$. Furthermore, we can take
$b=\oo$ as the base point for both $H_0$ and $H_1$ since, by
$(A3),$ $DH_i(\oo)=DG_1(\oo)$ is an isomorphism.  Now let us apply
Lemma \ref{generalized-homotopy}  choosing  as  path joining $(
0,\oo)$ with $(1,\oo)$ the path  $\gamma(t)=(t,\oo).$  It follows
then, that
\begin{equation}\label{varhom}
\text{deg}_{\oo}(H_0,\Omega_0',\oo) =\sigma(M)
\text{deg}_{\oo}(H_1,\Omega_1',\oo),
\end{equation}

where $M$ is the closed path of Fredholm operators given by
$M(t):=D_xH(t,\oo)= DH_t(\oo)$. On the other hand, Assumption
$(A3)$ implies that $\oo$ is the only solution of $H_i(\xx);
i=0,1$ which is a regular point  of $H_i$ since $DH_i(\oo)$ is an
isomorphism. By  definition of the base point degree
(\cite{Pej-Rab}) and Assumption $(A3)$, one has
\begin{equation*}
\text{deg}_{\oo}(H_0,\Omega'_0,\oo)=\text{deg}_{\oo}(H_1,\Omega'_1,\oo)=1,
\end{equation*}
which in turn implies, by \eqref{varhom}, that $\sigma(M)=1$. But,
by Lemma \ref{parity-lemma} and \eqref{w}, this contradicts our
assumption \eqref{w1}.\newline\indent In order to prove $[ii]$ we
observe that $[i]$ implies that the closure of $\cS$ in $S^1\times
X^+$ is a compact space and the closure of $\cC_0$ in this space
is a connected set intersecting both $\cB_0$ and $\cB_\infty.$  In
order to conclude the proof of $[ii]$ it is enough to use the
following slightly improved version of Whyburn's lemma:
\begin{proposition}[\cite{Al}, Proposition 5]\label{whyburn's-lemma}
Suppose $A$ and $B$ are closed and not separated in a compact
space $X.$ Then there exists a connected set $D\subset X\setminus
(A\cup B)$  such that $\bar D\cap A\neq \emptyset$ and $\bar D\cap
B\neq \emptyset.$
\end{proposition}
\end{proof}
\section{Example}
Now we are going to illustrate the content of Theorem
\ref{Theorem-A} formulated in Section 2 and the techniques
developed in this paper. Fix $0<\alpha< 1 $
and $\beta>1$. For
$\la=\exp(i\theta)$, $0\leq \theta\leq 2\pi$, we define $a\colon
S^1\ra GL(2)$ as follows
\begin{equation*}
a(\la)=a(\exp i\theta ):=
\begin{pmatrix}
 \displaystyle \alpha+(\beta-\alpha)\sin^2\left(\frac{\theta}{2}\right) &\displaystyle \frac{\alpha-\beta}{2}\sin(\theta) \\
   \displaystyle \frac{\alpha-\beta}{2}\sin(\theta) &  \displaystyle \alpha+(\beta-\alpha)\cos^2\left(\frac{\theta}{2}\right)
\end{pmatrix}.
\end{equation*}
Then we can consider  the  linear nonautonomous system ${\aa}
=(a_n (\la) ) \colon \Z\times S^1 \ra GL(2) $ defined by
\begin{equation}\label{constant}
a_n(\la) =
\begin{cases} a(\la)& \hif\; n\geq 0,
 \\ a(1)
 &\hif\; n<0.
\end{cases}
\end{equation}
 Since
independently of  $\la \in S^1$ the matrix $a(\la)$ has two
eigenvalues  $\alpha\in (0,1)$ and $\beta\in (1,\infty),$ the
system $\aa\colon \Z\times S^1 \ra GL(2)$ is asymptotically
hyperbolic. We will  apply our results to nonlinear perturbations
of $\aa\colon \Z\times S^1 \ra GL(2).$ We compute the asymptotic
stable bundles of $\aa$  at $\pm\infty\!:$
\begin{align*}
E^s(+\infty)&=\{(\la,u)\in S^1\times \R^2\mid
u=t(\cos(\theta/2),\sin(\theta/2)),
\la=\exp(i\theta),t\in\R\},\\
E^s(-\infty)&=\{(\la,u)\in S^1\times \R^2\mid u=(t,0),t\in\R\}.
\end{align*}
Thus $E^s(-\infty)$ is a trivial bundle and hence
$w_1(E^s(-\infty))=1.$ In order to compute $w_1(E^s(+\infty))$ we
notice that $v_\theta = (\cos(\theta/2),\sin(\theta/2))$ is a
basis for $ E^s_\theta(+\infty)$ which is the fiber of the
pullback $E' $ of $ E^s(+\infty)$ by the map $p\colon [0,2\pi] \ra
S^1$ defined by $p(\theta) = \exp(i \theta).$ Since $v_0=(1,0)$
and $v_{2\pi}=(-1,0),$ the determinant of the matrix $C$ arising
in \eqref{whitney} is $-1.$ Hence $w_1(E^s(+\infty))=-1 \neq
w_1(E^s(-\infty)).$

Let us consider the family $L$ of operators $L_\la$ defined by
\begin{equation*}
L_\la(\xx) (n)=
\begin{cases} x_{n+1}-a(\la)x_n & \hif\;  n\geq 0, \\ x_{n+1}-a(1)x_n & \hif\; n<0.
\end{cases}\end{equation*}
Then Remark \ref{ker} implies that $\Ker L_\la$ is isomorphic to $E^s(\la,+\infty)\cap E^u(\la,-\infty).$
But $E^u(-\infty)=\{(\la,u)\in S^1\times \R^2\mid u=(0,t),t\in\R\}$ and hence a nontrivial intersection
arises only for $\theta =\pi,$ i.e.,  $ \la =-1.$ Thus  $\Ker L_\la \neq \{\oo\}$ only if $\la=-1.$
 Theorem \ref{prop:ind} implies that $L_{\la}$ is a Fredholm operator of index:
$\ind(L_\la) =\dim E^s (\la,+\infty)-\dim E^s(\la,-\infty)=1-1=0.$
Hence we infer that $L_{\la}$ is an isomorphism for $\la\neq -1$, since $\Ker L_{\la}=\{\oo\}$, for $\la\neq -1$.
\newline\indent
Let $\hh\colon \Z\times S^1\times \R^2\ra\R^2$ be a continuous
family satisfying Assumption $(A1)$ and
\begin{enumerate}
\item[$(A2')$] $D_xh_n(\la,0) \xrightarrow[n\rightarrow\pm\infty]{}0$ uniformly
with respect to $\la\in S^1$;
\item[$(A3')$] for any $x\in \R^2$ and $\la\in S^1$,
$h_n(\la,x)\xrightarrow[n\rightarrow\pm\infty]{}
h^{\infty}_{\pm}(\la,x)$ (uniformly with respect to any bounded
set $B\subset \R^2$), and the following two difference equations
$x_{n+1}=a(\la)x_n+h^{\infty}_{\pm}(\la,x_n)$ admit only the
trivial solution $(x_n=0)_{n\in\Z}$, for all $\la\in S^1$;
\item[$(A4')$] for some  $\la_0\in S^1\setminus\{-1\}$ we have that
$D_x h_n(\la_0,0)=0$, for all $n\in\Z$, and the  difference
equation $x_{n+1}=a(\la_0)x_n+h_n(\la_0,x_n)$ admits only the
trivial solution $(x_n=0)_{n\in\Z}$.
\end{enumerate}
Now it is easily seen that whenever the nonlinear perturbation
$\hh$ verifies $(A1)$ and $(A2')$--$(A4')$ then the family
$\ff=\aa+\hh$ satisfies all the assumptions of Theorem
\ref{Theorem-A}. Therefore  it must have  a connected branch of
nontrivial homoclinic solutions joining $\cT_0$ with $\cT_\infty.$

\section*{Acknowledgements}
The first author is supported by  MIUR-PRIN 2009-Metodi
variazionali e topologici nei fenomeni nonlineari. The second
author is supported in part by Polish Scientific Grant N N201
395137.


\begin{thebibliography}{99}
\bibitem{Ab-Ma2} Abbondandolo A., Majer P.,
On the global stable manifold, Studia Math., 2006, 177(2),
113--131
\bibitem{Al}  J. C. Alexander, A primer on connectivity, In: Fixed Point Theory, Sherbrooke, June 2--21, 1980, Lecture Notes in
Math., 886, Springer, Berlin–New York, 1981, 455--483
\bibitem{At} Atiyah M.F., K-Theory, W.A. Benjamin, New York--Amsterdam, 1967
\bibitem{bach-Nar} Bachman G., Narici L., Functional Analysis, Dover, Mineola, 2000
\bibitem{Bar} Bartsch T., The global structure of the zero set of a family of semilinear
Fredholm maps, Nonlinear Analysis, 1991, 17(4), 313--331
\bibitem{Ba} Baskakov A. G., Invertibility and the Fredholm property of difference
operators, Mathematical Notes, 2000, 67(5-6), 690--698
\bibitem{Fi-Pej-88} Fitzpatrick P.M., Pejsachowicz J.,
The fundamental group of the space of linear Fredholm operators
and the global analysis of semilinear equations, In: Fixed Point
Theory and its Applications, Berkeley, August 4–-6, 1986, Contemp.
Math., 72, American Mathematical Society, Providence, 1988, 47–-87
\bibitem{Fi-Pej-Rab} Fitzpatrick P.M., Pejsachowicz J., Rabier P. J., The degree for proper $C^2$ Fredholm mappings
I, J. Reine Angew. Math., 1992, 427, 1--33
\bibitem{Hus} Husemoller D., Fibre Bundles, 2nd ed., Grad. Texts in Math., 20, Springer, New York--Heidelberg, 1975
\bibitem{Kat} Kato T., Perturbation Theory for Linear Operators, 2nd ed., Grundlehren Math. Wiss., 132, Springer, Berlin–New
York, 1976
\bibitem{Mo} Morris J.R., Nonlinear Ordinary and Partial Differential Equations on Unbounded Domains, PhD thesis, University
of Pittsburgh, 2005
\bibitem{Pejs-2} Pejsachowicz J., Bifurcation of homoclinics, Proc. Amer. Math. Soc., 2008, 136(1), 111--118
\bibitem{Pejs-3} Pejsachowicz J., Bifurcation of homoclinics of Hamiltonian systems, Proc. Amer. Math. Soc., 2008, 136(6), 2055--2065
\bibitem{Pejs-4} Pejsachowicz J., Topological invariants of bifurcation, $C^*$-algebras and elliptic theory II,
Trends Math., Birkhuser, Basel, 2008, 239--250
\bibitem{Pejs-1} Pejsachowicz J., Bifurcation of Fredholm maps I; Index bundle and bifurcation,
Topol. Methods Nonlinear Anal., 2011, 38, no. 1, 115--168
\bibitem{Pej-Rab} Pejsachowicz J., Rabier P. J., Degree theory for $C^1$-Fredholm mappings of index $0$,
J. Anal. Math., 1998, 76, 289--319
\bibitem{Pej-Ski} Pejsachowicz J., Skiba R.,  Topology and homoclinic trajectories
of discrete dynamical systems, preprint available at
http://arxiv.org/abs/1111.1402
\bibitem{Potz} P\"{o}tzsche C., Nonautonomus bifurcation of bounded solutions I: A
Lyapunov-Schmidt approach, Discrete Contin. Dyn. Syst. Ser. B,
2010, 14(2), 739--776
\bibitem{Potz-3} P\"otzsche C., Bifurcation Theory, Munich University of Technology, 2010, preprint available
at http://www-m12.
ma.tum.de/web/poetzsch/Christian$_{-}$Potzsche$_{-}$(Publications)/
Publications$_{-}$files/BifScript.pdf
\bibitem{Potz-1} P\"{o}tzsche C., Nonautonomous bifurcation of bounded solutions
II: A shovel bifurcation pattern, Discrete Contin. Dyn. Syst. Ser.
A, 2011, 31(3), 941--973
\bibitem{Potz-2} P\"{o}tzsche C., Nonautonomus continuation of bounded solutions,
Commun. Pure Appl. Anal., 2011, 10(3), 937--961
\bibitem{Ras} Rasmussen M., Towards a bifurcation theory for nonautonomous difference
equation, J. Difference Equ. Appl., 2006, 12(3--4), 297--312
\bibitem{Sa} Sacker R.J., The splitting index for linear differential systems,
J. Differential. Equations, 1979, 33(3), 368--405
\bibitem{Se-St} Secchi S., Stuart C. A., Global Bifurcation of homoclinic solutions of Hamiltonian
systems, Discrete Contin. Dyn. Syst., 2003, 9(6), 1493--1518
\end{thebibliography}
\end{document}